\documentclass[11pt,reqno]{amsart}

\usepackage[a4paper,margin=1.02in]{geometry}
\usepackage[T1]{fontenc}
\usepackage{lmodern}
\usepackage{amsmath,amssymb,amsthm,mathtools}
\usepackage{enumitem}
\usepackage{microtype}
\usepackage{xcolor,comment}
\usepackage[hidelinks]{hyperref}
\usepackage[nameinlink,noabbrev]{cleveref}

\allowdisplaybreaks
\setlist[enumerate]{leftmargin=2.4em,itemsep=0.25em,topsep=0.35em}
\setlist[itemize]{leftmargin=2.1em,itemsep=0.25em,topsep=0.35em}
\numberwithin{equation}{section}
\newtheorem{theorem}{Theorem}[section]
\newtheorem{proposition}[theorem]{Proposition}
\newtheorem{corollary}[theorem]{Corollary}
\newtheorem{lemma}[theorem]{Lemma}
\newtheorem{assumption}{Assumption}[section]
\theoremstyle{definition}
\newtheorem{definition}{Definition}[section]
\newtheorem{remark}{Remark}[section]

\crefname{assumption}{Assumption}{Assumptions}
\Crefname{assumption}{Assumption}{Assumptions}

\newcommand{\R}{\mathbb R}
\newcommand{\E}{\mathbb E}
\newcommand{\Pp}{\mathbb P}
\newcommand{\1}{\mathbf 1}
\newcommand{\cL}{\mathcal L}
\newcommand{\dd}{\,\mathrm d}
\newcommand{\e}{\mathrm e}
\newcommand{\supp}{\operatorname{supp}}
\newcommand{\doi}[1]{\href{https://doi.org/#1}{doi:#1}}

\def\beqlb{\begin{eqnarray}}\def\eeqlb{\end{eqnarray}}
\def\beqnn{\begin{eqnarray*}}\def\eeqnn{\end{eqnarray*}}
\def\ar{\!\!\!&}

\title[Invariant measures for interacting populations]
{{Invariant Measures for Interacting\\
Continuous-State Population Models}}

\author{Lina Ji}
\address{Faculty of Computational Mathematics and Cybernetics, Shenzhen MSU-BIT University, Shenzhen, China.}
\email{jiln@smbu.edu.cn}

\author{Jie Xiong}
\address{Department of Mathematics and Shenzhen International Center for Mathematics,
	Southern University of Science and Technology, Shenzhen 518055, China.}
\email{xiongj@sustech.edu.cn}

\author{Wen Xu }
\address{School of Mathematical Sciences, Peking University, Beijing 100871, China.}
\email{xuwen@math.pku.edu.cn}

\author{Xu Yang}
\address{School of Mathematics and Information Science, North Minzu University, Yinchuan, China.}
\email{xuyang@mail.bnu.edu.cn}

\thanks{This work was supported by
	the National Key R\&D Program of China (2022YFA1006102),
	the National Natural Science Foundation of China (12471418, 12595294, 12231002, 12301167, 12471135), the New Cornerstone Science Foundation (NCI202501), Guangdong Basic and Applied Basic Research Foundation  (2022A1515110986), and Shenzhen National Science Foundation  (20231128093607001).}
\date{August 8, 2026}

\hypersetup{
  pdftitle={Invariant Measures for Interacting Continuous-State Population Models},
  pdfauthor={Lina Ji, Jie Xiong, Wen Xu and Xu Yang}
}

\begin{document}

\begin{abstract}
	We study invariant probabilities for nonlinear two-population
	branching systems with competition, cooperation, or predator--prey
	interaction, driven by Brownian noise and spectrally positive
	stable jumps. Without immigration, every invariant probability
	of a conservative competitive or predator--prey system is
	concentrated on the boundary under either boundary convention.
	With nontrivial axis dynamics, coordinatewise continuation
	admits only the point mass at the origin.
	For cooperative systems, we give criteria for boundary
	concentration and interior stationarity, including a sharp
	criterion for bilinear cooperation with a common relative noise
	exponent. With immigration, explicit parameter conditions yield
	existence and stationary moment bounds.
	For square-root diffusions with positive immigration, these
	conditions ensure interior concentration and full support on
	the closed quadrant. Further lower bounds on immigration rates
	give uniqueness and convergence of time-averaged transition laws.
 
\end{abstract}


\maketitle

\section{Introduction and main results}

Continuous-state branching processes, introduced by
Ji\v{r}ina~\cite{Jirina1958}, describe the evolution of population
mass under independent reproduction and arise as scaling limits
of discrete branching processes; see Lamperti~\cite{Lamperti1967}.
In the subcritical regime, a continuous-state branching process converges to zero almost surely.
Continuous-state branching processes with immigration
were first studied by Kawazu and Watanabe~\cite{KawazuWatanabe1971}.
Dawson and Li~\cite{DawsonLi2006} provided constructions of continuous-state branching processes with immigration via stochastic equations as pathwise unique strong solutions, thereby establishing a framework for incorporating nonlinear rates and interactions.
The stationary behavior of continuous-state branching processes with immigration was investigated by
Pinsky~\cite{Pinsky1972} and further developed by Keller-Ressel
and Mijatovi\'c~\cite{KellerResselMijatovic2012}.
In the strictly subcritical case, the existence of a finite logarithmic moment of the immigration Lévy measure ensures convergence in distribution to a unique invariant probability; see Keller-Ressel and Mijatovi\'c \cite{KellerResselMijatovic2012}.

Stationary distributions have also been studied for one-dimensional
nonlinear branching models, which generally lack the classical
branching property.
Under suitable dissipativity and noise assumptions, Li and
Wang~\cite{LiWang2020} established existence and uniqueness of
an invariant probability and exponential convergence in the
\(L^1\)-Wasserstein distance and the total variation norm.
For continuous-state branching processes with immigration and
competition, Li et al.~\cite{LiLiWangZhou2025} proved exponential
ergodicity in a weighted total variation distance under
appropriate conditions. 
Chen et al.~\cite{ChenFangJiWang2026} established exponential
ergodicity in a weighted total variation distance for models
with immigration in L\'evy random environments, competition,
and catastrophes, under conditions balancing the branching,
environmental, and regulatory terms.

 Related results are available for stochastic
	Lotka--Volterra systems driven by environmental noise.
	Nguyen and Yin~\cite{NguyenYin2017} established conditions
	for a unique interior invariant probability and
	convergence in total variation for competitive systems.
	For mutualistic systems, Ji and
	Jiang~\cite{JiJiang2012} proved the existence of an
	ergodic stationary distribution under suitable conditions.
	For predator--prey models with a Beddington--DeAngelis
	functional response, Du et al.~\cite{DuNguyenYin2016}
	obtained conditions for a unique interior invariant
	probability and convergence in total variation.
	In the broader setting of stochastic Kolmogorov systems,
	Hening and Nguyen~\cite{HeningNguyen2018} developed
	criteria for coexistence and extinction, including
	sufficient conditions for a unique interior invariant
	probability and exponential convergence in total variation. Bao and Shao~\cite{BaoShao2016} established conditions
	for extinction, permanence, and the existence of stationary
	distributions in regime-switching predator--prey models.
We refer the reader to Bena\"{\i}m and Lobry \cite{BenaimLobry2016} and Hening et al. \cite{HeningNguyenNguyenWatts2023}
for related studies of stochastic
Lotka--Volterra systems
under random environmental switching
	and Feng, et al.~\cite{FengGrigorescuQuastel2004}
	for spatial predator--prey interactions.

In this paper, we study stationary behavior in two-dimensional
nonlinear branching systems with Lotka--Volterra-type interactions, both with and
without immigration. Depending on the signs of the
interaction coefficients, these systems describe competition,
cooperation, or predator--prey interaction.
Let \((\Omega,\mathcal F,\Pp)\) be a complete probability space. For $i = 1, 2$, let
\(B_i=(B_i(t))_{t\ge0}\) be Brownian motions and let
\(N_i(\dd s,\dd z,\dd u)\) be Poisson random measures on
\(\R_+\times(0,\infty)^2\) with intensities
\(\dd s\,\mu_i(\dd z)\,\dd u\), where
\[
\mu_i(\dd z)=
\frac{\alpha_i(\alpha_i-1)}
{\Gamma(\alpha_i)\Gamma(2-\alpha_i)}
z^{-1-\alpha_i}\1_{\{z>0\}}\dd z, \qquad \alpha_i\in(1,2).
\]
Let \(\widetilde N_i\) be the compensated measures of $N_i$.  Assume that
\(B_1,B_2,N_1,N_2\) are mutually independent, and let
\((\mathcal F_t)_{t\ge0}\) be the usual augmentation of the filtration
generated by these noises. 
We consider the following system:
\begin{equation}
	\left\{
	\begin{aligned}
		X_t={}&{x_0}+
		\int_0^t\bigl(\iota_1+a_1 X_s^{\theta_1}Y_s^{\kappa_1}-b_{10}X_s^{r_{10}}\bigr)\dd s
		+\int_0^t\sqrt{2b_{11}X_s^{r_{11}}}\dd B_1(s)
		\\
		&+\int_0^t\int_0^\infty\int_0^{b_{12}X_{s-}^{r_{12}}}
		z\,\widetilde N_1(\dd s,\dd z,\dd u),\\
		Y_t={}&{y_0}+
		\int_0^t\bigl(\iota_2+a_2 Y_s^{\theta_2}X_s^{\kappa_2}-b_{20}Y_s^{r_{20}}\bigr)\dd s
		+\int_0^t\sqrt{2b_{21}Y_s^{r_{21}}}\dd B_2(s)
		\\
		&+\int_0^t\int_0^\infty\int_0^{b_{22}Y_{s-}^{r_{22}}}
		z\,\widetilde N_2(\dd s,\dd z,\dd u).
	\end{aligned}
	\right.
	\label{eq:model}
\end{equation}
Here \((x_0,y_0)\in (0, \infty)^2\), \(a_i\in\R\), \(\kappa_i>0\), and
\(\iota_i,\theta_i,r_{ij},b_{ij}\ge0\), for \(i=1,2\) and
\(j=0,1,2\). The cases {\(a_1>0>a_2\)}, {\(a_1,a_2<0\)}, and
{\(a_1,a_2>0\)} describe predator-prey interaction, competition,
and cooperation, respectively. The constants \(\iota_i\) are
the immigration rates. If \(a_i=0\), the corresponding interaction
term is absent; whenever \(b_{ij}=0\), the corresponding drift
or noise term is omitted. A broader formulation allows interaction terms \(a_i f_i(x,y)\),
with nonnegative continuous \(f_i\), nonnegative coefficient
functions \(\gamma_{i,j}\) in the branching dynamics, and general
L\'evy measures for the jumps, under suitable well-posedness
assumptions. To simplify the presentation and obtain explicit
parameter conditions, we consider power coefficients and stable
L\'evy densities in the intensities of the driving Poisson random
measures.

 Extinction and boundary behavior
of related nonlinear branching systems were studied by
Ren et al.~\cite{RenXiongYangZhou2022},
Xiong et al.~\cite{XiongYangZhouCompeting,XiongYangZhouEnhancing},
and Ji et al.~\cite{JiXiongXuYang2026}.
For competitive or weakly cooperative Lotka--Volterra
diffusions, Cattiaux and M\'el\'eard~\cite{CattiauxMeleard2010}
established, under a balance condition, a unique
quasi-stationary distribution for the process killed when
either population becomes extinct.
Chen et al.~\cite{ChenLiWang2026} proved existence
and uniqueness of an invariant probability and uniform
exponential convergence in total variation for a branching
system with positive immigration and unidirectional
interactions, under
suitable conditions. 
Our main results concern invariant probabilities for bidirectionally
interacting systems with nonlinear branching rates, including
the distinction between boundary and interior stationarity
in the absence of immigration.

Write \(E=[0,\infty)^2\), \(E^\circ=(0,\infty)^2\), and
\(\partial E=E\setminus E^\circ\). Let \(B_b(E)\) and \(C_b(E)\) denote the spaces of bounded
Borel measurable and bounded continuous real-valued
functions on \(E\), respectively.
We write \(Z_t=(X_t,Y_t)\) for the population state.
For each Markov realization considered below, write
\(\Pp_{x_0,y_0}\) and \(\E_{x_0,y_0}\) for its law and
expectation when started from \((x_0,y_0)\in E\).
For an initial distribution \(\pi\), the corresponding
law and expectation are denoted by \(\Pp_\pi\) and
\(\E_\pi\). Functions on \(E\) are extended by zero
at the cemetery point whenever the process is killed.

We now summarize the main results.
Without immigration, we distinguish between freezing the
system at its first boundary hit and allowing the surviving
population to continue along its axis
(see \Cref{def:absorbed,def:coordinatewise}).
For conservative competitive and predator--prey systems,
every invariant probability is concentrated on \(\partial E\)
(see \Cref{thm:noimm-empirical,thm:competitive-noimm}).
Under coordinatewise continuation and nontrivial axis
dynamics, \(\delta_{(0,0)}\) is the unique invariant probability.
For cooperative systems, we give sufficient conditions
for boundary concentration and for the existence of a
nondegenerate invariant probability
(see \Cref{thm:cooperative-noimm-boundary,thm:cooperative-interior}). A highlight of the paper is \Cref{thm:cooperative-sharp},
which gives a necessary and sufficient criterion for
interior stationarity under bilinear cooperation with
a common relative noise exponent. This criterion includes the critical case, where the
interaction strengths and branching noise parameters
determine whether stationary coexistence is possible.

For models with immigration, we give explicit parameter
conditions ensuring conservativeness, existence of invariant
probabilities, and stationary moment bounds in all three
interaction regimes
(see \Cref{thm:imm-existence}).
Under additional restrictions on the noise exponents,
we obtain exact stationary balance identities
(see \Cref{cor:stationary-balance}).
For square-root diffusions satisfying these conditions,
positive immigration ensures that every invariant
probability is concentrated on \(E^\circ\) and has full
support on \(E\).
Further lower bounds on the immigration rates yield
uniqueness and weak convergence of the time-averaged
transition laws
(see \Cref{cor:sqrt-main}).

The proofs combine test functions, martingale arguments,
and boundary-hitting estimates.
Without immigration, occupation estimates follow from
powers chosen to cancel the interaction terms in
predator--prey systems and from the supermartingale
property of the total population in competitive systems.
For cooperative systems, we construct Lyapunov functions
combining negative powers near the boundary with a
weighted linear term at infinity.
With immigration, weighted linear Lyapunov functions
provide occupation estimates and stationary moment bounds.
The existence proofs combine the resulting tightness
estimates with the weak Feller property, established
through localization.
The critical criterion for bilinear cooperation further
uses stationary identities justified by concave truncation.
For square-root diffusions, the analysis also uses
localized comparison with Cox--Ingersoll--Ross diffusions,
interior ellipticity, and the support theorem to establish
boundary behavior, full support, and uniqueness.

\paragraph{\textbf{Notation.}}
{ For \(a,b\in\R\), write \(a\wedge b:=\min\{a,b\}\),
\(a\vee b:=\max\{a,b\}\), and \(a^+:=a\vee0\).
For an open set \(U\subset\R^2\), let \(C^2(U)\) denote the space
of twice continuously differentiable real-valued functions on \(U\).
Its subspace \(C_b^2(U)\) consists of functions that, together with
their first- and second-order partial derivatives, are bounded.
Similarly, \(C_b^\infty(U)\) denotes the space of smooth functions
whose derivatives of every order are bounded.
For vector- or matrix-valued functions, these conventions apply
componentwise.}
For sets \(A,D\subset\R^2\), the notation \(A\Subset D\) means that
\(\overline A\) is a compact subset of \(D\).  We use \(|\cdot|\) for
the Euclidean norm (and for absolute value on \(\R\)), write
{\(B(q,r):=\{w\in\R^2:|w-q|<r\}\)} for the open Euclidean ball with center \(q\) and radius
\(r\), and set
\[
 \operatorname{dist}(A,D):=\inf\{|a-d|:a\in A,\ d\in D\}.
\]
For a Borel measure \(\nu\), \(\supp\nu\) denotes its support relative
to the underlying state space. For probability measures \(\nu_n\) and \(\nu\) on \(E\),
we write \(\nu_n\Rightarrow\nu\) to denote weak convergence.

\subsection{Models without immigration}
\label{sec:noimm-results}

Throughout this subsection, set \(\iota_1=\iota_2=0\) in \eqref{eq:model}.  All passage times below are taken over the interval on
which the solution is defined, with
\(\inf\varnothing:=\infty\).
For \(R>0\), let
\[
\tau_R^-:=\inf\{t\ge0:X_t\wedge Y_t\le R\},
\qquad
\tau_R^+:=\inf\{t\ge0:X_t\vee Y_t\ge R\},
\]
and define the boundary and explosion times by
\[
\tau_0:=\lim_{R\downarrow0}\tau_R^-,
\qquad
\tau_\infty:=\lim_{R\to\infty}\tau_R^+,
\]
respectively.
For every \((x_0,y_0)\in E^\circ\), \eqref{eq:model}
admits a pathwise unique strong c\`adl\`ag solution on
\([0,\tau_0\wedge\tau_\infty)\), as follows from the
localization argument in the proof of
\Cref{prop:axis-realization}.
Boundary initial states and continuation on
\(\{\tau_0<\tau_\infty\}\) are specified below.

The model
\eqref{eq:model} is called \emph{source-conservative}
if
\begin{equation*}
 {\Pp_{x_0,y_0}(\tau_\infty<\infty,\ \tau_\infty\le\tau_0)=0,
 \qquad (x_0,y_0)\in E^\circ.}
\end{equation*}
On \(\{\tau_0<\tau_\infty\}\), define the boundary value by
\[
 Z_{\tau_0}:=\lim_{t\uparrow\tau_0}Z_t\in\partial E.
\]
Now we consider the following two boundary conventions,
which may lead to different classes of invariant probabilities.

The coordinatewise continuation requires the coefficients attached to an
extinct coordinate to vanish at that coordinate.  For the active coefficients we impose
\begin{equation}
 \theta_1,\theta_2>0,
 \qquad
 r_{ij}>0\quad\text{whenever }b_{ij}>0,
 \quad i=1,2,\quad j=0,1,2.
 \label{eq:boundary-compatible}
\end{equation}
Thus every active coefficient attached to an extinct coordinate vanishes at
zero, and no convention for \(0^0\) is needed. By \Cref{prop:axis-realization}, the following coordinatewise
	construction is well defined and yields
	a strong Markov process.

\begin{definition}
	\label{def:coordinatewise}
	Assume \eqref{eq:boundary-compatible} holds.
	We call \(Z=(Z_t)_{t\ge0}\), with \(Z_t=(X_t,Y_t)\),
	a \emph{coordinatewise process} (CWP)
	if it satisfies the following:
	Before \(\tau_0\wedge\tau_\infty\), it follows
	\eqref{eq:model}. On \(\{\tau_0<\tau_\infty\}\),
	if \(Y\) reaches zero first, then \(Y\) remains zero and
	the surviving coordinate solves
	\[
	\begin{aligned}
		X_t={}&X_{\tau_0}
		-\int_{\tau_0}^t b_{10}X_s^{r_{10}}\dd s
		+\int_{\tau_0}^t\sqrt{2b_{11}X_s^{r_{11}}}\dd B_1(s)\\
		&+\int_{\tau_0}^t\int_0^\infty
		\int_0^{b_{12}X_{s-}^{r_{12}}}
		z\,\widetilde N_1(\dd s,\dd z,\dd u).
	\end{aligned}
	\]
	If \(X\) reaches zero first, then \(X\) remains zero and
	the surviving coordinate solves
	\[
	\begin{aligned}
		Y_t={}&Y_{\tau_0}
		-\int_{\tau_0}^t b_{20}Y_s^{r_{20}}\dd s
		+\int_{\tau_0}^t\sqrt{2b_{21}Y_s^{r_{21}}}\dd B_2(s)\\
		&+\int_{\tau_0}^t\int_0^\infty
		\int_0^{b_{22}Y_{s-}^{r_{22}}}
		z\,\widetilde N_2(\dd s,\dd z,\dd u).
	\end{aligned}
	\]
	If both coordinates are zero, the process remains at the origin.
	On \(\{\tau_\infty\le\tau_0\}\), it is killed at
	\(\tau_\infty\). For a boundary initial state, set
	\(\tau_0=0\) and start the corresponding axis dynamics
	at time zero.
\end{definition}

\begin{definition}
	\label{def:absorbed}
We call \(Z^{\tau_0}=(Z_t^{\tau_0})_{t\ge0}\)
a \emph{source-stopped process} (SSP)
if it satisfies the following: for an initial state in
\(E^\circ\), it follows \eqref{eq:model} up to
\(\tau_0\wedge\tau_\infty\), remains at \(Z_{\tau_0}\)
after \(\tau_0\) on \(\{\tau_0<\tau_\infty\}\), and is
killed at \(\tau_\infty\) on \(\{\tau_\infty\le\tau_0\}\).
For an initial state \((x_0,y_0)\in\partial E\), set
\(\tau_0=0\) and \(Z_t^{\tau_0}=(x_0,y_0)\) for all
\(t\ge0\).
	\(Z_t^{\tau_0}=(x_0,y_0)\) for all \(t\ge0\).
\end{definition}

{ \begin{remark}
		Definition \ref{def:absorbed} does not require
	\eqref{eq:boundary-compatible}. When \eqref{eq:boundary-compatible} holds,
	\[
	Z_t^{\tau_0}=Z_{t\wedge\tau_0},
	\qquad t\ge0,
	\]
	where the CWP \(Z\) is the continuation in \Cref{def:coordinatewise},
	with the same killing convention.
\end{remark}
}

Denote the transition semigroups of the CWP \(Z\) and the SSP \(Z^{\tau_0}\)
by \((P_t)_{t\ge0}\) and \((P_t^{\tau_0})_{t\ge0}\),
respectively. Thus, for \((x_0,y_0)\in E\) and \(f\in B_b(E)\),
\[
 P_tf(x_0,y_0)=\E_{x_0,y_0}[f(Z_t)],
 \qquad
 P_t^{\tau_0}f(x_0,y_0)=\E_{x_0,y_0}[f(Z_t^{\tau_0})].
\]
In this subsection, \(P_t\) is used only when the CWP is defined. For \(T\ge1\), define the
corresponding time-averaged transition laws by
\[
 \Pi_T^{(x_0,y_0)}
 :=\frac1T\int_0^T P_t((x_0,y_0),\cdot)\dd t,
 \qquad
 \Pi_T^{\tau_0,(x_0,y_0)}
 :=\frac1T\int_0^T P_t^{\tau_0}((x_0,y_0),\cdot)\dd t.
\]

For any transition semigroup \((Q_t)_{t\ge0}\) considered
in this paper, its realization is called
\emph{conservative} if \(Q_t1=1\) for every \(t\ge0\),
equivalently, if the process is nonexplosive from every
initial state. The semigroup is called \emph{weak Feller}
if \(Q_tf\in C_b(E)\) for every \(t\ge0\) and \(f\in C_b(E)\).
A probability measure \(\pi\) on \(E\) is \emph{invariant}
if
\[
 \int_E Q_tf(x,y)\,\pi(\dd x,\dd y)
 =\int_E f(x,y)\,\pi(\dd x,\dd y),
 \qquad t\ge0,\quad f\in B_b(E).
\]
An invariant probability \(\pi\) is called
\emph{nondegenerate} if \(\pi(E^\circ)=1\).

Set
\beqlb\label{eqq}
 q_1:=1+\kappa_2-\theta_1,
 \qquad
 q_2:=1+\kappa_1-\theta_2.
\eeqlb

\begin{assumption}
\label{ass:dissipative}
{Assume that {\(a_1>0>a_2\)}.}
{For each \(i=1,2\), assume that \(0<q_i<\alpha_i\),
\(b_{i0}>0\), \(r_{i0}>1\), and}
\begin{equation}
	{
		{\left\{
		\begin{aligned}
			r_{i0}-1&>r_{i1}-2 &&\text{if }b_{i1}>0,\\
			r_{i0}-1&>r_{i2}-\alpha_i &&\text{if }b_{i2}>0.
		\end{aligned}
		\right.}
	}
	\label{eq:dissipative-dominance}
\end{equation}
If \(q_i>1\), assume in addition that
\beqnn
 q_i+r_{i1}-2\ge0\quad\text{if }b_{i1}>0,
 \qquad
 q_i+r_{i2}-\alpha_i\ge0\quad\text{if }b_{i2}>0.
\eeqnn 
\end{assumption}

\Cref{ass:dissipative} concerns the predator--prey regime
{\(a_1>0>a_2\)}. Its hypotheses coincide with those of
\cite[Theorem~1.2]{JiXiongXuYang2026} when the positive
suppression coefficient in that paper is replaced by \(-a_2>0\).
We use that result only for the SSP in the predator--prey regime. The conclusion of CWP also requires the axis
estimate below.

\begin{theorem}
\label{thm:noimm-empirical}
Assume that \(\iota_1=\iota_2=0\) and
\Cref{ass:dissipative} holds.
\begin{enumerate}[label=\textup{(\roman*)}]
\item A probability \(\pi\) is invariant for the SSP \(Z^{\tau_0}\) if and only if \(\pi(\partial E)=1\).
\item If \eqref{eq:boundary-compatible} holds, then
\(\delta_{(0,0)}\) is the unique invariant probability for
the CWP \(Z\). If, in addition,
\((P_t)_{t\ge0}\) is weak Feller, then for every $(x_0,y_0)\in E$,
\[
 \Pi_T^{(x_0,y_0)}\Rightarrow\delta_{(0,0)}
 \qquad\text{as }T\to\infty.
\]
\end{enumerate}
\end{theorem}

For a related competitive model, nonexplosion was established
in \cite[Lemma~1(i)]{XiongYangZhouCompeting} under the restrictions
\(r_{i0}\ge1\), \(r_{i1}\ge2\), and \(r_{i2}\ge\alpha_i\)
with \(b_{i1}+b_{i2}>0\) for
\(i=1,2\).
Under \eqref{eq:boundary-compatible}, the following theorem
establishes nonexplosion and tightness of the time-averaged
transition laws under both boundary conventions, without the above restrictions on \(r_{ij}\) or the assumption
\(b_{i1}+b_{i2}>0\).

\begin{theorem}
\label{thm:competitive-noimm}
Assume that \(\iota_1=\iota_2=0\), \(a_1,a_2<0\), and
\eqref{eq:boundary-compatible} holds.
A probability \(\pi\) is invariant for the SSP \(Z^{\tau_0}\) if and only if \(\pi(\partial E)=1\), and every
invariant probability for the CWP \(Z\) is
concentrated on \(\partial E\).
If, in addition,
\begin{equation}
 b_{10}+b_{11}+b_{12}>0,
 \qquad b_{20}+b_{21}+b_{22}>0,
 \label{eq:axis-nontrivial}
\end{equation}
then \(\delta_{(0,0)}\) is the unique invariant probability
for the CWP \(Z\). If \((P_t)_{t\ge0}\) is also
weak Feller, then for $(x_0,y_0)\in E$,
\[
 \Pi_T^{(x_0,y_0)}\Rightarrow\delta_{(0,0)}
 \qquad\text{as }T\to\infty.
\]
\end{theorem}

Cooperative systems may admit invariant probabilities
concentrated on \(E^\circ\).
Their nonexplosion requires a separate argument.
For the cooperative results below, we assume that
\begin{equation}
	b_{i0}>0,\qquad
	\theta_i,\kappa_i\ge1,\qquad
	r_{ij}\ge1\ \text{if }b_{ij}>0,
	\quad i,j=1,2.
	\label{eq:cooperative-regularity}
\end{equation}
\begin{assumption}
	\label{ass:cooperative-dissipation}
	The interaction coefficients satisfy \(a_1,a_2>0\), and
	\begin{equation}
		\frac{\theta_1}{r_{10}}+\frac{\kappa_1}{r_{20}}\le1,
		\qquad
		\frac{\kappa_2}{r_{10}}+\frac{\theta_2}{r_{20}}\le1.
		\label{eq:cooperative-powers}
	\end{equation}
	If both equalities hold in \eqref{eq:cooperative-powers}, we further assume that
	\begin{equation}
		\left(\frac{a_1}{b_{10}}\right)^{\kappa_2/r_{10}}
		\left(\frac{a_2}{b_{20}}\right)^{\kappa_1/r_{20}}<1.
		\label{eq:cooperative-critical}
	\end{equation}
\end{assumption}

The following three theorems describe the stationary behavior
of cooperative systems without immigration.
\Cref{thm:cooperative-noimm-boundary} gives sufficient conditions
under which every invariant probability is concentrated on
the boundary.
\Cref{thm:cooperative-interior} gives sufficient conditions
for the existence of an invariant probability concentrated
on \(E^\circ\), together with stationary moment bounds.
For bilinear cooperation with a common relative noise exponent,
\Cref{thm:cooperative-sharp} gives a necessary and sufficient
criterion for the existence of such an interior invariant
probability, including the critical case.

\begin{theorem}
\label{thm:cooperative-noimm-boundary}
Assume \eqref{eq:cooperative-regularity} and
\Cref{ass:cooperative-dissipation} hold with
\(\iota_1=\iota_2=0\). 
Suppose that one of the following conditions holds:
\begin{enumerate}[label=\textup{(\roman*)}]
\item The interaction powers satisfy
\begin{equation}
 \frac{\theta_1}{r_{10}}+\frac{\kappa_1}{r_{20}}
 =\frac{\kappa_2}{r_{10}}+\frac{\theta_2}{r_{20}}=1.
 \label{eq:cooperative-critical-powers}
\end{equation}
\item \(b_{11}+b_{12}>0\), \(b_{21}+b_{22}>0\), and, for
at least one \(i\in\{1,2\}\),
\begin{equation}
 (r_{i1}<2\ \text{if}\ b_{i1}>0)
 \quad\text{or}\quad
 (r_{i2}<\alpha_i\ \text{if}\ b_{i2}>0).
 \label{eq:cooperative-accessible-noise}
\end{equation}
\item \(b_{11}+b_{12}>0\), \(b_{21}+b_{22}>0\), and
\begin{equation*}
 \max_{i=1,2}
 \min\bigl\{r_{i1}-2\ \text{if}\ b_{i1}>0,\;
             r_{i2}-\alpha_i\ \text{if}\ b_{i2}>0\bigr\}
 <(\theta_1+\kappa_1)\wedge(\theta_2+\kappa_2)-1.
\end{equation*}
\end{enumerate}
Then a probability \(\pi\) is invariant for the SSP \(Z^{\tau_0}\) if and only if \(\pi(\partial E)=1\). For the
CWP \(Z\), \(\delta_{(0,0)}\) is the unique
invariant probability.
\end{theorem}

When \(r_{10}=r_{20}=2\), \(\theta_i=\kappa_i=1\) for \(i=1,2\), and \(a_1,a_2>0\),
\Cref{ass:cooperative-dissipation} reduces to
\(a_1a_2<b_{10}b_{20}\), and condition \textup{(i)}
of \Cref{thm:cooperative-noimm-boundary} holds.
Condition \textup{(ii)} of  \Cref{thm:cooperative-noimm-boundary} includes
square-root branching noise.

\begin{theorem}
\label{thm:cooperative-interior}
Assume \eqref{eq:cooperative-regularity} and
\Cref{ass:cooperative-dissipation} hold with
\(\iota_1=\iota_2=0\). For \(i=1,2\), set
\[
 r_i:=\min\bigl\{r_{i0}-1,\;
          r_{i1}-2\ \text{if }b_{i1}>0,\;
          r_{i2}-\alpha_i\ \text{if }b_{i2}>0\bigr\}.
\]
Suppose that \(r_i>\theta_i-1\) for \(i=1,2\) and that either
\begin{enumerate}[label=\textup{(\roman*)}]
\item  \((r_1-\theta_1+1)(r_2-\theta_2+1) >\kappa_1\kappa_2\); or
\item \( (r_1-\theta_1+1)(r_2-\theta_2+1) =\kappa_1\kappa_2\) and
\[
 \left(\frac{a_1}{{b_1}}\right)^{r_2-\theta_2+1}
 \left(\frac{a_2}{{b_2}}\right)^{\kappa_1}>1,
\]
where, for \(i=1,2\),
\beqnn
 {b_i}:= b_{i0}\1_{\{r_{i0}-1= r_i \}}
 +(r_i+1)b_{i1}\1_{\{r_{i1}-2= r_i\}}+\frac{\Gamma(r_i+\alpha_i)}
 {\Gamma(r_i+1)\Gamma(\alpha_i)}
 b_{i2}\1_{\{r_{i2}-\alpha_i=r_i\}}.
\eeqnn
\end{enumerate}
Then the CWP and the SSP have a common nondegenerate
invariant probability \(\pi_\circ\) satisfying, for some
\(\eta,\zeta>0\),
\[
 \int_{E^\circ}
 \bigl(x^{r_{10}}+y^{r_{20}}+x^{-\eta}+y^{-\zeta}\bigr)
 \,\pi_\circ(\dd x,\dd y)<\infty.
\]
\end{theorem}

When \(\theta_i=\kappa_i=1\) for \(i=1,2\),
\Cref{thm:cooperative-interior}\textup{(i)} applies provided that
\beqlb\label{eq1005} 
 r_{i0}>2,\qquad
 r_{i1}>3\ \text{if }b_{i1}>0,\qquad
 r_{i2}>\alpha_i+1\ \text{if }b_{i2}>0,
 \qquad i=1,2.
\eeqlb
 Under \eqref{eq1005}, both inequalities in
\eqref{eq:cooperative-powers} hold strictly. And \eqref{eq1005} ensure the existence of an interior invariant
probability. When the two coordinates share a common relative
noise exponent, the following theorem gives a necessary and
sufficient condition.

\begin{theorem}
\label{thm:cooperative-sharp}
Assume \eqref{eq:cooperative-regularity} holds with
\(\iota_1=\iota_2=0\), \(a_1,a_2>0\),
\(\theta_i=\kappa_i=1\), \(r_{i0}>2\), and
\(b_{i1}+b_{i2}>0\), \(i=1,2\).
Suppose that, for some constant \(\beta\in\R\),
\[
 r_{i1}=2+\beta\quad\text{if }b_{i1}>0,
 \qquad
 r_{i2}=\alpha_i+\beta\quad\text{if }b_{i2}>0,
 \qquad i=1,2.
\]
When \(\beta=1\), define
\[
 \Psi_i(p):=b_{i1}p(p+1)
   +b_{i2}\frac{\Gamma(p+\alpha_i)}
                    {\Gamma(p)\Gamma(\alpha_i)},
 \qquad p>0,\quad i=1,2,
\]
and let \(\widehat q_1,\widehat q_2\in(0,1)\) be the unique solutions of
\beqlb\label{q1q2} 
 \Psi_1(\widehat q_1)=a_2(1-\widehat q_1),
 \qquad
 \Psi_2(\widehat q_2)=a_1(1-\widehat q_2).
\eeqlb
For both the CWP \(Z\) and the SSP \(Z^{\tau_0}\),
a nondegenerate invariant probability exists if and only if
\[
 \beta>1
 \quad\text{or}\quad
 \bigl(\beta=1\ \text{and}\ \widehat q_1+\widehat q_2>1\bigr).
\]
Otherwise, the invariant probabilities of the
SSP \(Z^{\tau_0}\) are exactly the probabilities concentrated on
\(\partial E\), whereas \(\delta_{(0,0)}\) is the unique
invariant probability of the CWP \(Z\).
\end{theorem}

\subsection{Models with immigration}
\label{sec:imm-results}

{For the results in this subsection, we allow nonnegative
immigration rates and assume}
\begin{equation}
	\iota_i\ge0,\quad b_{i0}>0,\quad r_{i0}>1,\quad
	\theta_i,\kappa_i\ge1,\quad
	r_{ij}\ge1\ \text{if }b_{ij}>0,\quad i,j=1,2.
	\label{eq:imm-regularity}
\end{equation}
{The following assumption controls the drift at infinity.}
\begin{assumption}
	\label{ass:weighted-dissipation}
	One of the
	following conditions holds:
	\begin{enumerate}[label=\textup{(\roman*)}]
		\item \(a_1,a_2\le0\).
		
		\item \(a_1>0\), \(a_2\le0\), and
		\begin{equation}
			\frac{\theta_1}{r_{10}}+\frac{\kappa_1}{r_{20}}
			-\1_{\{a_2<0\}}
			\left(
			\frac{\kappa_2}{r_{10}}+\frac{\theta_2}{r_{20}}-1
			\right)^+
		 \left(\frac{\theta_1}{\kappa_2}
             \wedge\frac{\kappa_1}{\theta_2}\right)
			\le1.
			\label{eq:mixed-parameter-condition}
		\end{equation}
		\item \Cref{ass:cooperative-dissipation} holds.
	\end{enumerate}
\end{assumption}
{By \Cref{prop:imm-wellposed}, equation \eqref{eq:model}
admits a pathwise unique nonnegative strong solution
\(Z=(X,Y)\) from every initial state in \(E\), for arbitrary
signs of \(a_1,a_2\).} When \(\iota_1=\iota_2=0\), it is the
CWP in \Cref{def:coordinatewise}.
The results below concern this solution, including at zero
immigration. We retain the notation \(P_t\) for its transition
semigroup, with killing at \(\tau_\infty\), and
\(\Pi_T^{(x_0,y_0)}\) for its time-averaged transition laws,
defined as in \Cref{sec:noimm-results}.

We first establish existence and moment bounds for invariant probabilities.
\begin{theorem}
\label{thm:imm-existence}
Assume \eqref{eq:imm-regularity} and
\Cref{ass:weighted-dissipation} hold. Then the transition semigroup $(P_t)_{t \ge 0}$
of \eqref{eq:model} admits an invariant probability.
Moreover, there exist constants \(c, c_1>0\) and \(C_0\ge0\),
depending only on the parameters of \eqref{eq:model},
such that every invariant probability \(\pi\) satisfies
\begin{equation}
	\left\{
	\begin{aligned}
		&\int_E\bigl(b_{10}x^{r_{10}}+c_1b_{20}y^{r_{20}}
		-a_1x^{\theta_1}y^{\kappa_1}
		-c_1a_2x^{\kappa_2}y^{\theta_2}\bigr)
		\pi(\dd x,\dd y)
		\le\iota_1+c_1\iota_2,\\
		&\int_E\bigl(x^{r_{10}}+y^{r_{20}}\bigr)
		\pi(\dd x,\dd y)
		\le\frac{\iota_1+c_1\iota_2+C_0}{c}<\infty.
	\end{aligned}
	\right.
	\label{eq:imm-moment}
\end{equation}
\end{theorem}

The stationary drift inequality becomes an equality under the following noise-moment condition.
\begin{corollary}
 \label{cor:stationary-balance}
 Under the assumptions of \Cref{thm:imm-existence}, suppose in addition that
 \begin{equation}
  r_{ij}\le r_{i0}\quad\text{if }b_{ij}>0,
  \qquad i,j=1,2.
  \label{eq:noise-moment-order}
 \end{equation}
 Then 
 \begin{enumerate}[label=\textup{(\roman*)}]
 \item The constants \(c,c_1,C_0\) in \Cref{thm:imm-existence} can be chosen
 so that every invariant probability \(\pi\) satisfies
 \begin{equation}
  \int_E\bigl(b_{10}x^{r_{10}}+c_1b_{20}y^{r_{20}}
  -a_1x^{\theta_1}y^{\kappa_1}-c_1a_2x^{\kappa_2}y^{\theta_2}\bigr)
  \,\pi(\dd x,\dd y)=\iota_1+c_1\iota_2.
  \label{eq:stationary-balance}
 \end{equation}
 \item If \(a_1>0>a_2\) and
 \begin{equation}
  \theta_1=\kappa_2,\qquad \theta_2=\kappa_1,
  \label{eq:balanced-interactions}
 \end{equation}
 then \Cref{ass:weighted-dissipation}\textup{(ii)} holds automatically,
one may take \(c_1=a_1/(-a_2)\), and every invariant probability \(\pi\) satisfies
 \[
  b_{10}\int_E x^{r_{10}}\pi(\dd x,\dd y)
  +\frac{a_1}{-a_2}b_{20}\int_E y^{r_{20}}\pi(\dd x,\dd y)
  =\iota_1+\frac{a_1}{-a_2}\iota_2.
 \]
 \end{enumerate}
\end{corollary}

For interaction terms with the same sign, the two coordinate equations give separate stationary balance identities.
\begin{corollary}
\label{cor:individual-interaction-balances}
Assume \eqref{eq:imm-regularity}, \eqref{eq:noise-moment-order} and
\Cref{ass:weighted-dissipation} hold. If \(a_1a_2>0\), then with all integrals finite, every
invariant probability \(\pi\) satisfies the following:
\begin{equation}
 \left\{
 \begin{aligned}
 b_{10}\int_E x^{r_{10}}\pi(\dd x, \dd y)
 &=\iota_1+a_1\int_E x^{\theta_1}y^{\kappa_1}\pi(\dd x, \dd y),\\
 b_{20}\int_E y^{r_{20}}\pi(\dd x, \dd y)
 &=\iota_2+a_2\int_E x^{\kappa_2}y^{\theta_2}\pi(\dd x, \dd y).
 \end{aligned}
 \right.
 \label{eq:individual-interaction-balances}
\end{equation}
\end{corollary}

\subsection{Stationary coexistence for square-root diffusions}
\label{sec:sqrt-results}

Set \(b_{12}=b_{22}=0\), \(r_{11}=r_{21}=1\), and
\(b_{11},b_{21}>0\). Equation \eqref{eq:model} becomes
\begin{equation}
 \left\{
 \begin{aligned}
 X_t={}&x_0+\int_0^t
 (\iota_1+a_1X_s^{\theta_1}Y_s^{\kappa_1}
          -b_{10}X_s^{r_{10}})\dd s
 +\int_0^t\sqrt{2b_{11}X_s}\dd B_1(s),\\
 Y_t={}&y_0+\int_0^t
 (\iota_2+a_2Y_s^{\theta_2}X_s^{\kappa_2}
          -b_{20}Y_s^{r_{20}})\dd s
 +\int_0^t\sqrt{2b_{21}Y_s}\dd B_2(s).
 \end{aligned}
 \right.
 \label{eq:sqrt-system}
\end{equation}

\begin{corollary}
\label{cor:sqrt-main}
Assume \eqref{eq:imm-regularity} and
\Cref{ass:weighted-dissipation} hold for \eqref{eq:sqrt-system} with \(\iota_1,\iota_2>0\). Then
\begin{enumerate}[label=\textup{(\roman*)}]
\item The process $Z$ admits an invariant probability.
Every invariant probability \(\pi\) is concentrated on \(E^\circ\),
has support \(E\), and satisfies \eqref{eq:imm-moment}
and \eqref{eq:stationary-balance}.
\item If
\begin{equation}
 \iota_1\ge b_{11},\qquad \iota_2\ge b_{21},
 \label{eq:Feller-thresholds}
\end{equation}
then the invariant probability \(\pi\) is unique and
\(\Pi_T^{(x_0,y_0)}\Rightarrow\pi\) as \(T\to\infty\)
for every \((x_0,y_0)\in E\).
\end{enumerate}
\end{corollary}

{For \eqref{eq:sqrt-system} with \(r_{10}=r_{20}=2\)
and \(\theta_i=\kappa_i=1\) for \(i=1,2\),}
positive immigration ensures the existence of an invariant
probability concentrated on \(E^\circ\) for arbitrary competitive
interaction strengths, and for cooperative interactions satisfying
\(a_1a_2<b_{10}b_{20}\).
In both cases, the invariant probability is unique under
\eqref{eq:Feller-thresholds}.

The rest of the paper is organized as follows.
\Cref{sec:preliminaries} records the generator and the common
drift, nonexplosion, and stationary moment estimates used in the proofs of the main results.
\Cref{sec:noimm-proofs} studies the stationary behavior of the models without immigration.
{
\Cref{sec:imm-proofs} proves existence of invariant probabilities
and stationary balance identities for models with immigration.
} \Cref{sec:sqrt-proofs} proves \Cref{cor:sqrt-main}.
The appendix contains the required existence, uniqueness, and
boundary continuation results.

\section{Preliminaries}
\label{sec:preliminaries}

The estimates in this section will be used for the models with and without immigration.
We therefore allow arbitrary \(a_1,a_2\in\R\) and \(\iota_1,\iota_2\ge0\).
The operator \(\cL\) associated with \eqref{eq:model} is given, for \((x,y)\in E^\circ\), by
\beqlb\label{eq:generator}
 \cL f(x,y)\ar=\ar
 \bigl(\iota_1+a_1x^{\theta_1}y^{\kappa_1}-b_{10}x^{r_{10}}\bigr)f_x'(x,y)
 +\bigl(\iota_2+a_2y^{\theta_2}x^{\kappa_2}-b_{20}y^{r_{20}}\bigr)f_y'(x,y)
 \cr
 \ar\ar +b_{11}x^{r_{11}}f_{xx}''(x,y)
 +b_{21}y^{r_{21}}f_{yy}''(x,y)\cr
 \ar\ar +b_{12}x^{r_{12}}\int_0^\infty
 [f(x+z,y)-f(x,y)-zf_x'(x,y)]\,\mu_1(\dd z)\cr
 \ar\ar +b_{22}y^{r_{22}}\int_0^\infty
 [f(x,y+z)-f(x,y)-zf_y'(x,y)]\,\mu_2(\dd z)
\eeqlb
for \(f\in C^2(E^\circ)\) such that the above integrals are finite.
Here subscripts denote partial derivatives. Fix \((x_0,y_0)\in E^\circ\) and
\(0<\varepsilon<x_0\wedge y_0\le x_0\vee y_0<N\). When $\cL f$ is bounded on $[\epsilon, N]^2$,
by It\^o's formula, the process
\[
 f(Z_{t\wedge\tau_\varepsilon^-\wedge\tau_N^+})
 -f(x_0,y_0)
 -\int_0^{t\wedge\tau_\varepsilon^-\wedge\tau_N^+}
       \cL f(Z_s)\dd s
\]
is a local martingale. If, in addition, \(f\ge0\), by Fatou's lemma and the localizing  sequence of stopping times for a local martingale, we obtain
\beqlb\label{eq:localized-dynkin-inequality}
 \E_{x_0,y_0}f(Z_{t\wedge\tau_\varepsilon^-\wedge\tau_N^+}) \le f(x_0,y_0)
 +\E_{x_0,y_0}\int_0^{t\wedge\tau_\varepsilon^-\wedge\tau_N^+}
       \cL f(Z_s)\dd s.
 \eeqlb
In particular, for every \(f\in C_b^2(E^\circ)\), $\cL f$ is bounded on $[\epsilon, N]^2$, the process
\begin{equation}
t \mapsto  f(Z_{t\wedge\tau_\varepsilon^-\wedge\tau_N^+})
 -f(x_0,y_0)
 -\int_0^{t\wedge\tau_\varepsilon^-\wedge\tau_N^+}
       \cL f(Z_s)\dd s \quad \text{is a martingale}.
 \label{eq:localized-dynkin-martingale}
\end{equation}

{Without immigration, for} the CWP $Z$ under \eqref{eq:boundary-compatible},
we also use \(\cL\) on each open axis for the one-dimensional
operator acting on the restriction of the test function to that
axis.  At the absorbing origin, we set \(\cL f(0,0)=0\).

Under \eqref{eq:imm-regularity}, the generator
\eqref{eq:generator} is also well defined on \(E\) for any
test function \(f\) that admits a \(C^2\) extension to an
open neighbourhood of \(E\) and satisfies the above integral
bound for every compact \(K\subset E\).
For such \(f\), the coefficients and \(\cL f\) are bounded
on each square \([0,N]^2\). Localizing at \(\tau_N^+\)
therefore yields the same local-martingale identity,
together with the corresponding inequality when \(f\ge0\)
and the martingale identity when \(f\) is bounded.

 For \(c_1>0\) and \((x,y)\in E\), let 
\beqlb\label{SVW}
S(x,y):=x+c_1y,\quad  V_0(x,y):=1+S(x,y), \quad W_0(x,y):=x^{r_{10}}+y^{r_{20}}.
\eeqlb 
The weight \(c_1\) is chosen separately in each application.
By \eqref{eq:generator},
\beqlb\label{eq:weighted-generator-expansion}
 \cL S(x,y)= \cL V_0(x,y) = \iota_1+c_1\iota_2-b_{10}x^{r_{10}}-c_1b_{20}y^{r_{20}}+a_1x^{\theta_1}y^{\kappa_1}
 +c_1a_2x^{\kappa_2}y^{\theta_2}.
 \eeqlb

The next three lemmas provide common drift and moment estimates.
They also apply to the cooperative model without immigration:
under \eqref{eq:cooperative-regularity} and
\Cref{ass:cooperative-dissipation},
\eqref{eq:cooperative-powers} implies
\[
 r_{10}> \theta_1\vee \kappa_2 \ge1,
 \qquad r_{20}> \kappa_1\vee \theta_2 \ge1.
\]
Thus \eqref{eq:imm-regularity} and
\Cref{ass:weighted-dissipation}\textup{(iii)} hold { under \eqref{eq:cooperative-regularity} and
	\Cref{ass:cooperative-dissipation}} when
\(\iota_1=\iota_2=0\).

\begin{lemma}
\label{lem:parameter-drift}
Assume that \eqref{eq:imm-regularity} and
\Cref{ass:weighted-dissipation} hold. Then there exist constants
\(c, c_1>0\) and \(C_0\ge0\) such that the functions \(V_0, W_0\) defined by \eqref{SVW}
satisfy
\begin{equation}
 \cL V_0(x,y)\le \iota_1+c_1\iota_2+C_0
 -cW_0(x,y),
 \qquad (x,y)\in E.
 \label{eq:weighted-dissipation}
\end{equation}
\end{lemma}

\begin{proof}
It suffices to consider \(x,y>0\). Let
\[
 u=x^{r_{10}},\quad v=y^{r_{20}},\qquad
 p=\frac{\theta_1}{r_{10}},\quad q=\frac{\kappa_1}{r_{20}},\quad
 r=\frac{\kappa_2}{r_{10}},\quad s=\frac{\theta_2}{r_{20}}.
\]
 By \eqref{eq:weighted-generator-expansion}, 
\beqlb\label{eq1005a}
 \cL V_0(x,y)=\iota_1+c_1\iota_2-b_{10}u-c_1b_{20}v
             +a_1u^pv^q+c_1a_2u^rv^s.
\eeqlb
 If \(a_1,a_2\le0\), every \(c_1>0\) is admissible
and \(C_0=0\). If \(a_1>0>a_2\) and
\eqref{eq:balanced-interactions} holds, the interaction terms
cancel for \(c_1=a_1/(-a_2)\), and again \(C_0=0\).
It remains to consider \(a_1>0\). Suppose first that \(a_2<0\), and set
\[
 \lambda=
 \begin{cases}
 (p/r)\wedge(q/s),&r+s>1,\\
 0,&r+s\le1.
 \end{cases}
\]
By \eqref{eq:mixed-parameter-condition},
\[
 p-\lambda r\ge0,\qquad q-\lambda s\ge0,\qquad
 (p-\lambda r)+(q-\lambda s)+\lambda\le1.
\]
By the above, we have $p -\lambda r < 1$. In fact, 
we have $p -\lambda r \le 1 - \lambda - (q - \lambda s) \le 1 - \lambda < 1$ if $\lambda > 0$; for the case of $\lambda = 0$, since \(q=\kappa_1/r_{20}>0\) by \eqref{eq:imm-regularity}, we have $p \le 1 - q < 1$. By
Young's inequality, one obtains
\beqnn 
 a_1u^pv^q
 \ar=\ar a_1u^{p-\lambda r}v^{q-\lambda s}(u^rv^s)^\lambda
\le a_1u^{p-\lambda r}(1+v+u^rv^s)^{q-\lambda s+\lambda}
\cr
 \ar\le\ar a_1u^{p-\lambda r}(1+v+u^rv^s)^{1-p+\lambda r} \le  \frac{b_{10}}2u+K(1+v+u^rv^s)
\eeqnn 
for some \(K>0\). Taking
\(c_1\ge\max\{2K/b_{20},K/|a_2|\}\) in
\eqref{eq:weighted-generator-expansion}, \eqref{eq1005a} yields
\beqlb\label{eq1005b} 
 \cL V_0(x,y)\le\iota_1+c_1\iota_2+K
 -\frac{b_{10}}2u-\frac{c_1b_{20}}2v.
\eeqlb
When \(a_2=0\), we have
\(p+q\le1\) and \(p<1\) by \eqref{eq:mixed-parameter-condition}. The same argument gives
\(a_1u^pv^q\le(b_{10}/2)u+K(1+v)\), so
\(c_1\ge2K/b_{20}\) yields { the same bound form \eqref{eq1005b}.}
Suppose now that \(a_1,a_2>0\). By
\Cref{ass:cooperative-dissipation}, \(p+q\le1\) and \(r+s\le1\).
For all \(h,k>0\) with \(h+k<1\) and every \(\varepsilon>0\),
\beqlb\label{eq1005c} 
 u^hv^k\le(u+v)^{h+k}\le\varepsilon(u+v)+C_\varepsilon.
\eeqlb
If \(p+q<1\) and \(r+s<1\), take \(c_1=1\).
Applying \eqref{eq1005c} to the last two terms in
\eqref{eq1005a} gives
\[
 \cL V_0(x,y)\le\iota_1+c_1\iota_2+C_0
 -\frac{b_{10}}2u-\frac{c_1b_{20}}2v
\]
for some constant \(C_0\ge0\).
If \(p+q=1>r+s\), by Young's inequality, it gives
	\beqlb\label{eq1006a} 
a_1u^pv^q\le(b_{10}/4)u+Kv
\eeqlb 
for some \(K>0\).
Choose \(c_1\ge4K/b_{20}\). Since \(r+s<1\), then by \eqref{eq1005c},
\beqlb\label{eq1006}
 c_1a_2u^rv^s\le\frac{b_{10}}4u+\frac{c_1b_{20}}4v+C_0
\eeqlb
 for some \(C_0\ge0\).
By \eqref{eq1006a}, \eqref{eq1006} and \eqref{eq1005a}, it proves \eqref{eq:weighted-dissipation}.
The case \(r+s=1>p+q\) follows similarly by using
\(a_2u^rv^s\le Ku+(b_{20}/4)v\) and choosing
\(0<c_1\le b_{10}/(4K)\).
Finally, suppose that \(p+q=r+s=1\), so \(p,q,r,s\in(0,1)\).
Set
\[
 t_0=(a_1/b_{10})^{1/q},\qquad
 c_1=\frac{b_{10}q}{a_2r}t_0^{1-r},\qquad
 F(t)=b_{10}t+c_1b_{20}-a_1t^p-c_1a_2t^r.
\]
Then by \eqref{eq:cooperative-critical} and \eqref{eq:imm-regularity}, it yields $F''(t)>0, F'(t_0)=0$ and $F(t_0)=c_1(b_{20}-a_2t_0^r)>0.$
Hence \(F(t)\ge F(t_0)>0\) for all \(t\ge t_0\). Since
\(F(t)/(1+t)\to b_{10}>0\) as \(t\to\infty\), $c:=\inf_{t\ge0}\frac{F(t)}{1+t}>0.$
By \eqref{eq1005a},
\[
 \cL V_0(x,y)=\iota_1+c_1\iota_2-vF(u/v)
 \le\iota_1+c_1\iota_2-c(u+v).
\]
{ By continuity, the bound holds on \(E\), so
\eqref{eq:weighted-dissipation} follows with \(C_0=0\).}
\end{proof}

{The next two lemmas concern the nonnegative
solution of \eqref{eq:model}. At zero immigration, this is
the CWP. 
\begin{lemma}
	\label{lem:imm-lyapunov}
	Assume that \eqref{eq:imm-regularity} and
	\Cref{ass:weighted-dissipation} hold.
	Let \(c_1,c>0\) and \(C_0\ge0\) satisfy
	\eqref{eq:weighted-dissipation}.
	Then the nonnegative solution \(Z\) of \eqref{eq:model}
	is nonexplosive, and there exists a constant \(C>0\) such that,
	for every \((x_0,y_0)\in E\),
	\begin{equation}
		\E_{x_0,y_0}V_0(Z_t)
		\le e^{-t}V_0(x_0,y_0)+C(1-e^{-t}),
		\qquad t\ge0,
		\label{eq:V-first-moment}
	\end{equation}
	and
	\begin{equation}
		\E_{x_0,y_0}\int_0^T
		\bigl(\iota_1+c_1\iota_2+C_0-\cL V_0(Z_s)\bigr)\dd s
		\le  V_0(x_0,y_0)+(\iota_1+c_1\iota_2+C_0)T,
		\qquad T>0,
		\label{eq:imm-occupation-bound}
	\end{equation}
		where \(V_0\) is the function defined in \eqref{SVW}. Moreover, for every compact \(K\subset E\) and \(t>0\),
	\[
	\lim_{R\to\infty}\sup_{(x_0,y_0)\in K}
	\Pp_{x_0,y_0}(\tau_R^+\le t)=0.
	\]
\end{lemma}

\begin{proof}
For \(N>0\), set $\tau_N^{V_0}:=\inf\{t\ge0:\ V_0(Z_t)\ge N\}.$
Since \(r_{10},r_{20}>1\), by \eqref{eq:weighted-dissipation}, there exists a constant \(C>0\) such that
\begin{equation}
 \cL V_0+V_0\le C
 \qquad\text{on }E.
 \label{eq:linear-Foster}
\end{equation}
Fix \((x_0,y_0)\in E\). By \eqref{eq:model}, for $0\le t<\tau_\infty,$
\beqlb\label{Vito}
 V_0(Z_t)= V_0(x_0,y_0)+\int_0^t\cL V_0(Z_s)\dd s+M_t, 
\eeqlb
where \(M=(M_t)_{t\ge0}\) is the local martingale given by
\beqlb\label{M} 
 M_t\ar=\ar \int_0^t\sqrt{2b_{11}X_s^{r_{11}}}\dd B_1(s)
 +c_1\int_0^t\sqrt{2b_{21}Y_s^{r_{21}}}\dd B_2(s)  +\int_0^t\int_0^\infty\int_0^{b_{12}X_{s-}^{r_{12}}}
 z\,\widetilde N_1(\dd s,\dd z,\dd u)\cr
 \ar\ar +c_1\int_0^t\int_0^\infty\int_0^{b_{22}Y_{s-}^{r_{22}}}
 z\,\widetilde N_2(\dd s,\dd z,\dd u).
\eeqlb
For \(N> V_0(x_0,y_0)\), set $M_t^{(N)}:=\int_0^{t\wedge\tau_N^{V_0}}\e^s\dd M_s.$
It is easy to see that \(\E_{x_0,y_0}M_t^{(N)}=0\).
By integration by parts, we have
\beqnn
 \e^{t\wedge\tau_N^{V_0}}V_0(Z_{t\wedge\tau_N^{V_0}})
 = V_0({x_0,y_0})+\int_0^{t\wedge\tau_N^{V_0}}\e^s
 \bigl(V_0+\cL  V_0\bigr)(Z_s)\dd s +M_t^{(N)}.
\eeqnn 
Using \eqref{eq:linear-Foster} and taking expectations, we obtain
\begin{equation}
 \E_{{x_0,y_0}}\left[
 \e^{t\wedge\tau_N^{V_0}}V_0(Z_{t\wedge\tau_N^{V_0}})
 \right]
 \le V_0({x_0,y_0})+C(\e^t-1).
 \label{eq:exp-V-stopped}
\end{equation}
On \(\{\tau_N^{V_0}\le t\}\), we have $V_0(Z_{t\wedge\tau_N^{V_0}})\ge N.$
It follows that
\begin{equation}
 \Pp_{{x_0,y_0}}(\tau_N^{V_0}\le t)
 \le\frac{V_0({x_0,y_0})+C(\e^t-1)}{N}.
 \label{eq:nonexplosion-prob}
\end{equation}
Since
\beqlb\label{star} 
 1+(1\wedge c_1)(x\vee y)
 \le V_0(x,y)
 \le 1+(1+c_1)(x\vee y),
\eeqlb
we have \(\tau_N^{V_0}\uparrow\tau_\infty\).
Letting \(N\to\infty\) in \eqref{eq:nonexplosion-prob}
gives \(\Pp_{x_0,y_0}(\tau_\infty\le t)=0\) for every
\(t>0\). Thus the solution to \eqref{eq:model} is nonexplosive.
By \eqref{star}, we obtain $
 \{\tau_R^+\le t\}
 \subset\{\tau_{1+(1\wedge c_1)R}^{V_0}\le t\}.$
By \eqref{eq:nonexplosion-prob}, for every compact \(K\subset E\)
and all sufficiently large \(R\), we have
\[
 \sup_{(x_0,y_0)\in K}\Pp_{x_0,y_0}(\tau_R^+\le t)
 \le
 \frac{\sup_{(x_0,y_0)\in K}V_0(x_0,y_0)
 +C(\e^t-1)}{1+(1\wedge c_1)R}
 \rightarrow0\qquad\text{as }R\to\infty.
\]
Letting \(N\to\infty\) in \eqref{eq:exp-V-stopped}, using Fatou's
lemma, and dividing by \({\e^t}\), it yields
\[
 \E_{{x_0,y_0}}V_0(Z_t)
 \le\e^{-t}V_0({x_0,y_0})
 +C(1-\e^{-t}),
 \qquad t\ge0.
\]
This proves \eqref{eq:V-first-moment}.
Finally, by stopping \eqref{Vito} at \(T\wedge\tau_N^{V_0}\)
and taking expectations, we have
\begin{align*}
 &\E_{{x_0,y_0}}V_0(Z_{T\wedge\tau_N^{V_0}})
 +\E_{{x_0,y_0}}\int_0^{T\wedge\tau_N^{V_0}}
 \bigl(\iota_1+c_1\iota_2+C_0-\cL V_0(Z_s)\bigr)\dd s\\
 &\hspace{25mm}=V_0({x_0,y_0})+
 \left(\iota_1+c_1\iota_2+C_0\right)
 \E_{x_0,y_0}(T\wedge\tau_N^{V_0}).
\end{align*}
By \eqref{eq:weighted-dissipation},
\(\iota_1+c_1\iota_2+C_0-\cL V_0\ge0\).
Letting \(N\to\infty\) and applying Fatou's lemma, we obtain
\eqref{eq:imm-occupation-bound}.
The result follows.
\end{proof}
}

\begin{lemma}
	\label{lem:common-stationary-moments}
	Assume that \eqref{eq:imm-regularity} and
	\Cref{ass:weighted-dissipation} hold.
	Let \(c_1,c>0\) and \(C_0\ge0\)
	satisfy \eqref{eq:weighted-dissipation}.
	If \(\pi\) is an invariant probability for the
	nonnegative solution $Z$ of \eqref{eq:model}, then
	\[
	\left\{
	\begin{aligned}
		&\int_E\bigl(\iota_1+c_1\iota_2+C_0-\cL V_0(x, y)\bigr)\pi(\dd x, \dd y)
		\le \iota_1+c_1\iota_2+C_0,\\
		&\int_E W_0(x,y) \pi(\dd x, \dd y)
		\le \frac{\iota_1+c_1\iota_2+C_0}{c}<\infty.
	\end{aligned}
	\right.
	\]
\end{lemma}

\begin{proof}
 To avoid assuming any moment
	condition on \(\pi\), we first use the bounded functions
	\[
	f_n(x,y):=n\bigl(1-e^{- V_0(x,y)/n}\bigr),
	\qquad n\ge1.
	\]
The function \(v\mapsto n(1-e^{-v/n})\) is concave.
By \eqref{eq:generator}, it is easy to check that
	\[
	\cL f_n(x, y)\le e^{-V_0(x, y)/n}\cL V_0(x, y).
	\]
For each fixed \(n\), both \(\cL f_n\) and
\(e^{- V_0/n}\cL V_0\) are bounded on \(E\).
	By \Cref{lem:imm-lyapunov}, the solution $Z$ is nonexplosive.
	Since \(f_n\) and \(\cL f_n\) are bounded, localized
	It\^o's formula and bounded convergence yield, for \(t>0\),
	\[
	\E_\pi f_n(Z_t)-\E_\pi f_n(Z_0)
	=\int_0^t\E_\pi\cL f_n(Z_s)\dd s.
	\]
Hence, by stationarity, we obtain
	\[
	0=\int_E\cL f_n\,\dd\pi
	\le\int_E e^{-V_0/n} \cL V_0(x, y) \,\dd\pi.
	\]
Since \(0<e^{- V_0/n}\le1\), it follows that
\beqlb\label{eq:stationary-truncated-drift-bound}
\ar\ar \int_E  {e^{-V_0(x,y)/n}}
 \bigl(\iota_1+c_1\iota_2+C_0- {\cL V_0(x,y)}\bigr)\pi(\dd x,\dd y)\cr
 \ar\ar\qquad\le(\iota_1+c_1\iota_2+C_0)
 \int_E e^{- V_0(x,y)/n}\pi(\dd x,\dd y)
 \le\iota_1+c_1\iota_2+C_0.
\eeqlb
	By \eqref{eq:weighted-dissipation}, we have 
	\beqlb\label{eq2} 
	\iota_1+c_1\iota_2+C_0-\cL  V_0(x, y)
	\ge c W_0(x,y) \ge0.
	\eeqlb
 Letting \(n\to\infty\) in \eqref{eq:stationary-truncated-drift-bound}
and applying Fatou's lemma, we obtain
	\begin{equation}
	\int_E
	\bigl(\iota_1+c_1\iota_2+C_0-\cL V_0(x, y)\bigr)\pi(\dd x, \dd y)
	\le \iota_1+c_1\iota_2+C_0.
	\label{eq:stationary-integrated-drift-bound}
\end{equation}
	Combining \eqref{eq2} and \eqref{eq:stationary-integrated-drift-bound}, we obtain
	\[
	\int_E W_0(x,y) \pi(\dd x, \dd y)
	\le\frac{\iota_1+c_1\iota_2+C_0}{c}<\infty.
	\]
	The result follows.
\end{proof}

{
We will also use the following fact about time-averaged
transition laws.

\begin{lemma}
\label{lem:time-average-limits}
Let \(D=E\) or \(E^\circ\), and let \((Q_t)_{t\ge0}\)
be the transition semigroup of a conservative weak Feller
Markov process with c\`adl\`ag paths in \(D\).
For \(z\in D\), define
\[
\nu_T^z:=\frac1T\int_0^T Q_s(z,\cdot)\dd s.
\]
Then every weak subsequential limit of \(\nu_T^z\) as
\(T\to\infty\) is invariant for \((Q_t)_{t\ge0}\).
If \(\{\nu_T^z:T\ge1\}\) is tight on \(D\), an invariant
probability exists. If, in addition, the invariant probability
is unique, say \(\pi\), then
\(\nu_T^z\Rightarrow\pi\) as \(T\to\infty\).
\end{lemma}

\begin{proof}
We use the Krylov--Bogoliubov argument; see
\cite[Theorem~3.1.1, p.~21]{DaPratoZabczyk1996}.
For \(t\ge0\) and \(f\in C_b(D)\), the semigroup property gives
\beqnn
 \left|\int_D Q_tf\,\dd\nu_T^z-\int_D f\,\dd\nu_T^z\right|
 =\frac1T\left|
 \int_T^{T+t}Q_sf(z)\dd s-\int_0^tQ_sf(z)\dd s
 \right| \le\frac{2t\|f\|_\infty}{T}.
 \eeqnn
If \(\nu_{T_k}^z\Rightarrow\pi\) for some \(T_k\to\infty\),
the weak Feller property allows passage to the limit,
yielding
\[
 \int_D Q_tf\,\dd\pi=\int_D f\,\dd\pi.
\]
Thus \(\pi\) is invariant. Tightness and Prokhorov's theorem
give such a subsequential limit. If the invariant probability
is unique, every weak subsequential limit equals \(\pi\). Then by tightness, we have  \(\nu_T^z\Rightarrow\pi\). The result follows.
\end{proof}
}

\section{Models without immigration}
\label{sec:noimm-proofs}
Throughout this section, \(\iota_1=\iota_2=0\).
We begin with the exclusion of interior invariant
probabilities when \(a_2<0\).

\begin{lemma}
\label{lem:interior-exclusion}
{Assume that \(a_2<0\).} Let \((Q_t)_{t\ge0}\) denote either \((P_t^{\tau_0})_{t\ge0}\), provided
the SSP \(Z^{\tau_0}\) is conservative, or
\((P_t)_{t\ge0}\), provided \eqref{eq:boundary-compatible} holds
and the CWP \(Z\) is conservative.  Set \(f(x,y):=\e^{-y}\).
Then
\[
 Q_tf\ge f\quad\text{on }E,\quad t\ge0,
 \qquad
 Q_tf>f\quad\text{on }E^\circ,\quad t>0.
\]
{{Moreover,} every invariant probability \(\pi\)
for \((Q_t)_{t\ge0}\) satisfies \(\pi(\partial E)=1\).}
\end{lemma}

\begin{proof}
Observe that
\[
 f_x'(x,y)=f_{xx}''(x,y)=0,\qquad
 f_y'(x,y)=-\e^{-y},\qquad
 f_{yy}''(x,y)=\e^{-y},
\]
and
\(f(x,y+z)-f(x,y)-zf_y'(x,y)
=\e^{-y}(\e^{-z}-1+z)\).
Since
\[
 {\int_0^\infty(\e^{-z}-1+z)\,\mu_2(\dd z)
 =\frac1{\Gamma(\alpha_2)}},
\]
then it follows from \eqref{eq:generator} that, for \((x,y)\in E^\circ\),
\begin{equation}
 \begin{aligned}
 \cL f(x,y)
  =\e^{-y}\left[
 -a_2y^{\theta_2}x^{\kappa_2}
 +b_{20}y^{r_{20}}+b_{21}y^{r_{21}}
 +\frac{b_{22}}{\Gamma(\alpha_2)}y^{r_{22}}\right] \ge -a_2\e^{-y}y^{\theta_2}x^{\kappa_2}>0.
 \end{aligned}
 \label{eq:strict-positive-generator}
\end{equation}
We first consider the SSP $Z^{\tau_0}$ started from
\((x_0,y_0)\in E^\circ\).
For every sufficiently large integer \(k\),
\eqref{eq:localized-dynkin-martingale} and \(\cL f\ge0\) imply that $(f(Z_{t\wedge\tau_{1/k}^-\wedge\tau_k^+}))_{t\ge0}$
is a bounded submartingale.
By conservativeness and \Cref{def:absorbed}, for every fixed \(t\ge0\),
\[
 \lim_{k\to\infty}
 f(Z_{t\wedge\tau_{1/k}^-\wedge\tau_k^+})
 =f(Z_t^{\tau_0})
 \qquad\text{almost surely}.
\]
Since \(0<f\le1\), by dominated convergence in the conditional submartingale inequality, one sees that 
\((f(Z_t^{\tau_0}))_{t\ge0}\) is a bounded submartingale.
For the CWP, \eqref{eq:boundary-compatible}
ensures that all coefficients attached to an extinct
coordinate vanish at zero, so that keeping this coordinate
at zero is consistent with \eqref{eq:model}.
The preceding argument applies up to \(\tau_0\).
On the \(x\)-axis, \(f\equiv1\), while on the open
\(y\)-axis, since \(\kappa_2>0\),
\begin{equation}
\cL f(0,y)=e^{-y}\left[
b_{20}y^{r_{20}}+b_{21}y^{r_{21}}
+\frac{b_{22}}{\Gamma(\alpha_2)}y^{r_{22}}
\right]\ge0.
\label{eq:axis-laplace-generator}
\end{equation}
Localization along the axes and the strong Markov property
from \Cref{prop:axis-realization} therefore show that
\((f(Z_t))_{t\ge0}\) is a bounded
submartingale.
Thus, under either convention,
\begin{equation}
 Q_tf(x_0,y_0)\ge f(x_0,y_0),
 \qquad (x_0,y_0)\in E,\quad t\ge0.
 \label{eq:laplace-subinvariant}
\end{equation}
Fix \((x_0,y_0)\in E^\circ\), \(t>0\),
\(0<\varepsilon<x_0\wedge y_0\), and \(N>x_0\vee y_0\).
{Right continuity gives}
\(\Pp_{x_0,y_0}(\tau_\varepsilon^-\wedge\tau_N^+>0)=1\).
Moreover, by \eqref{eq:strict-positive-generator}, we have $c_0:=\min_{(x,y)\in[\varepsilon,N]^2}\cL f(x,y)>0.$
It follows from \eqref{eq:localized-dynkin-martingale} that
\beqnn
 \E_{x_0,y_0}f(Z_{t\wedge\tau_\varepsilon^-\wedge\tau_N^+})
 -f(x_0,y_0) \ar=\ar \E_{x_0,y_0}\int_0^{t\wedge\tau_\varepsilon^-\wedge\tau_N^+}
               \cL f(Z_s)\dd s\cr
 \ar\ge\ar c_0\E_{x_0,y_0}
        (t\wedge\tau_\varepsilon^-\wedge\tau_N^+)>0.
\eeqnn 
By optional sampling for the bounded submartingale, one sees that
\beqlb\label{eq0911} 
 Q_tf(x_0,y_0)
 \ge\E_{x_0,y_0}f(Z_{t\wedge\tau_\varepsilon^-\wedge\tau_N^+})
 >f(x_0,y_0).
\eeqlb
{By \eqref{eq:laplace-subinvariant} and \eqref{eq0911},
\(Q_tf-f\) is nonnegative on \(E\) and strictly positive on
\(E^\circ\) for every \(t>0\).}  If \(\pi\) is invariant, then $\int_E(Q_tf-f)\dd\pi=0,$
which forces \(\pi(E^\circ)=0\). The result follows.
\end{proof}

The preceding lemma gives the following classification.

\begin{proposition}
\label{prop:noimm-classification}
{Assume that \(\iota_1=\iota_2=0\) and {\(a_2<0\)}.}
\begin{enumerate}[label=\textup{(\roman*)}]
\item If the SSP $Z^{\tau_0}$ is conservative, then a
probability \(\pi\) is invariant for \((P_t^{\tau_0})\) if and only if
\(\pi(\partial E)=1\).
\item Assume \eqref{eq:boundary-compatible} and that the
CWP $Z$ is conservative.
{Then every invariant probability \(\pi\) for
\((P_t)\) satisfies \(\pi(\partial E)=1\).}  If, in addition,
\eqref{eq:axis-nontrivial} holds,
then \(\delta_{(0,0)}\) is its unique invariant probability.
\end{enumerate}
\end{proposition}

\begin{proof}
\noindent\textup{(i)}
By \Cref{def:absorbed}, we have $P_t^{\tau_0}((x_0,y_0),\cdot)=\delta_{(x_0,y_0)}$ for
every $(x_0,y_0)\in\partial E$ and $t\ge0.$
If \(\pi(\partial E)=1\), then, for every \(h\in B_b(E)\),
\[
 \int_E P_t^{\tau_0}h\dd\pi
 =\int_{\partial E}h(u,v)\pi(\dd u,\dd v)
 =\int_Eh\dd\pi.
\]
Thus \(\pi\) is invariant.  Conversely,
\Cref{lem:interior-exclusion} shows that every invariant probability
\(\pi\) satisfies \(\pi(\partial E)=1\).

\textup{(ii)}
Let \(\pi\) be invariant for \((P_t)_{t\ge0}\).
By \Cref{lem:interior-exclusion}, we have \(\pi(\partial E)=1\). Assume now that \eqref{eq:axis-nontrivial} holds,
and let \(f\) be as in \Cref{lem:interior-exclusion}.
By \eqref{eq:laplace-subinvariant}, \(P_tf\ge f\) on \(E\).
By \eqref{eq:axis-laplace-generator} and \eqref{eq:axis-nontrivial}, we have \(\cL f(0,y)>0\) for every \(y>0\).
Similar to the argument in \eqref{eq0911}, we obtain
\beqlb\label{eq0911a}
 P_tf(x_0,y_0)>f(x_0,y_0)\quad\text{if }y_0>0,\qquad
 P_tf(x_0,0)=f(x_0,0)=1,
 \qquad t>0.
\eeqlb
{By invariance of \(\pi\),}
\(\int_E(P_tf-f)\dd\pi=0\). By \eqref{eq0911a}, we have 
\(\pi([0,\infty)\times\{0\})=1\).
Set \(g(x,y):=\e^{-x}\). Similar to \eqref{eq0911a},
\[
 P_tg(x_0,0)>g(x_0,0)\quad\text{if }x_0>0,\qquad
 P_tg(0,0)=g(0,0)=1,\qquad t>0.
\]
Since \(\pi([0,\infty)\times\{0\})=1\), it yields
\[
 0=\int_{(0,\infty)\times\{0\}}
       \bigl(P_tg(u,0)-g(u,0)\bigr)\pi(\dd u,\dd v).
\]
Thus \(\pi((0,\infty)\times\{0\})=0\) and
\(\pi=\delta_{(0,0)}\). Hence \(\delta_{(0,0)}\) is invariant.
\end{proof}

\subsection{The predator--prey model}

Throughout this subsection, \Cref{ass:dissipative} holds;
in particular, \(a_1>0>a_2\).  
We use the Lyapunov functions from
\cite[Theorem~1.2]{JiXiongXuYang2026}.
Define
\[
 H(x,y):=\frac{x^{q_1}}{q_1}
 +\frac{a_1}{-a_2}\frac{y^{q_2}}{q_2},
 \qquad
 {W(x,y):=x^{q_1+r_{10}-1}+y^{q_2+r_{20}-1},}
\]
where $q_i$ is given by \eqref{eqq}.
Here {\(a_1/(-a_2)>0\)} and \(q_1,q_2>0\). Hence
\(H\ge0\) on \(E\) with equality only at $(0, 0)$.
Moreover,
\[
 \inf_{\substack{(x,y)\in E\\x\vee y\ge N}}H(x,y)
 \ge\min\left\{\frac{N^{q_1}}{q_1},
                 \frac{a_1}{-a_2}\frac{N^{q_2}}{q_2}\right\}
 \rightarrow\infty
\]
as $N \rightarrow \infty.$
The interaction terms in \eqref{eq:model} cancel, giving
\[
 a_1x^{\theta_1}y^{\kappa_1}H_x'
 +a_2y^{\theta_2}x^{\kappa_2}H_y'
 =a_1x^{\kappa_2}y^{\kappa_1}
  -a_1x^{\kappa_2}y^{\kappa_1}=0.
\]
Since \(q_i+r_{i0}-1>0\) for \(i=1,2\),
\(W\) has compact sublevel sets in \(E\).
The following lemma follows from
\cite[Theorem~1.2]{JiXiongXuYang2026}. We state the result without proof.
\begin{lemma}
\label{lem:predator-source-occupation}
Under \Cref{ass:dissipative}, the SSP $Z^{\tau_0}$ is
conservative, and there exist constants \(C,c>0\) such that
\begin{equation}
 \cL H(x,y)\le C-cW(x,y),
 \qquad (x,y)\in E^\circ,
 \label{eq:Foster-H}
\end{equation}
and
\begin{equation}
 {\frac1T\E_{x_0,y_0}\int_0^{T\wedge\tau_0}W(Z_s)\dd s
 \le \frac Cc+\frac{H(x_0,y_0)}{cT},
 \qquad (x_0,y_0)\in E^\circ,\quad T>0.}
 \label{eq:source-occupation}
\end{equation}
\end{lemma}

We next extend the above Lyapunov estimate to the CWP $Z$.

\begin{lemma}
\label{lem:cw-occupation}
Suppose that \Cref{ass:dissipative} and
\eqref{eq:boundary-compatible} hold.
{Then the CWP $Z$ is conservative. Moreover, there exist constants \(C,c>0\) such that, for every
\((x_0,y_0)\in E\) and \(t\ge0\),}
\begin{equation}
 {\E_{x_0,y_0}H(Z_t)
 +c\E_{x_0,y_0}\int_0^t W(Z_s)\dd s
 \le H(x_0,y_0)+Ct.}
 \label{eq:cw-global-Foster}
\end{equation}
Consequently, \(\{\Pi_T^{(x_0,y_0)}:T\ge1\}\) is tight
for every \((x_0,y_0)\in E\).
\end{lemma}

\begin{proof}
By \Cref{lem:predator-source-occupation}, the model
\eqref{eq:model} is source-conservative. By
\Cref{prop:axis-realization}, $\mathbb{P}_{x_0, y_0}(\tau_\infty = \infty) = 1$ for every $(x_0, y_0) \in E$. Then the
CWP $Z$ is conservative. For \((x_0,y_0)\in E^\circ\), apply
\eqref{eq:localized-dynkin-inequality} to \(H\), use
\eqref{eq:Foster-H}, and let \(\varepsilon\downarrow0\) and
\(N\to\infty\). By Fatou's lemma and monotone convergence, we have
\beqlb\label{eq:cw-preboundary-Foster}
		\E_{x_0,y_0}H(Z_{t\wedge\tau_0})
		+c\E_{x_0,y_0}\int_0^{t\wedge\tau_0}
		W(Z_s)\dd s \le H(x_0,y_0)+C\E_{x_0,y_0}(t\wedge\tau_0).
	\eeqlb
In particular,
\beqnn 
\E_{x_0,y_0}\!\left[
H(Z_{t\wedge\tau_0})
\1_{\{\tau_0\le t\}}\right]\le H(x_0,y_0)+Ct.
\eeqnn
By \eqref{eq:generator} and
\cite[Lemma~A.3]{JiXiongXuYang2026}, for \(x>0\),
\begin{equation}
\begin{aligned}
 \cL H(x,0)
 = -b_{10}x^{q_1+r_{10}-1}
 +(q_1-1)b_{11}x^{q_1+r_{11}-2}+\frac{(q_1-1)\Gamma(\alpha_1-q_1)}
 {\Gamma(\alpha_1)\Gamma(2-q_1)}
 b_{12}x^{q_1+r_{12}-\alpha_1}.
\end{aligned}
\label{eq:predator-axis-generator}
\end{equation}
If \(q_1\le1\), the last two terms on the right-hand side of
\eqref{eq:predator-axis-generator} are nonpositive.
If \(q_1>1\), by \Cref{ass:dissipative} and
\eqref{eq:dissipative-dominance}, we have $
0\le q_1+r_{11}-2<q_1+r_{10}-1$ if $b_{11}>0,$
and $0\le q_1+r_{12}-\alpha_1<q_1+r_{10}-1$ if
$b_{12}>0.$
Since \(b_{10}>0\), there exist constants  {\(C_1\ge0\) and \(c_1>0\)}
such that
\beqnn 
 {\cL H(x,0)\le C_1-c_1x^{q_1+r_{10}-1}=C_1-c_1W(x,0),\qquad x>0.}
\eeqnn
Similarly, there exist constants \(C_2\ge0\) and \(c_2>0\) such that
\beqnn 
 {\cL H(0,y)\le C_2-c_2y^{q_2+r_{20}-1}=C_2-c_2W(0,y),\qquad y>0.}
\eeqnn
Enlarge \(C\) to \(\max\{C,C_1,C_2\}\) and decrease \(c\) to
\(\min\{c,c_1,c_2\}\).
For an initial state on either open axis, apply It\^o's formula
after stopping the surviving coordinate on \([1/k,k]\), and let
\(k\to\infty\). Fatou's lemma and monotone convergence give
\eqref{eq:cw-global-Foster} for boundary initial states.
At \((0,0)\), the estimate is immediate because the process is constant.

On \(\{\tau_0\le t\}\), the process reaches the boundary
at time \(\tau_0\). Applying the strong Markov property
at \(\tau_0\) and the boundary estimate over the remaining
time \(t-\tau_0\), we obtain
\[
	 \E_{x_0,y_0}\left[
	\1_{\{\tau_0\le t\}}
	\left\{H(Z_t)
	+c\int_{t\wedge\tau_0}^t
	W(Z_s)\dd s\right\}\right] 
	 \le
	\E_{x_0,y_0}\left[
	\1_{\{\tau_0\le t\}}
	H(Z_{t\wedge\tau_0})
	\right]
	+C\E_{x_0,y_0}(t-\tau_0)^+.
\]
On \(\{\tau_0>t\}\), we have \(t\wedge\tau_0=t\), so
the occupation integral above vanishes and $
H(Z_t)
=H(Z_{t\wedge\tau_0}).$
Adding the contribution from \(\{\tau_0>t\}\) gives
\[
\begin{aligned}
	 \E_{x_0,y_0}H(Z_t)
	+c\E_{x_0,y_0}\int_{t\wedge\tau_0}^t
	W(Z_s)\dd s \le
	\E_{x_0,y_0}H(Z_{t\wedge\tau_0})
	+C\E_{x_0,y_0}(t-\tau_0)^+.
\end{aligned}
\]
Adding the occupation integral up to \(t\wedge\tau_0\)
and using \eqref{eq:cw-preboundary-Foster}, we obtain
\[
\begin{aligned}
	&\E_{x_0,y_0}H(Z_t)
	+c\E_{x_0,y_0}\int_0^t
	W(Z_s)\dd s\\
	&\qquad\le H(x_0,y_0)
	+C\E_{x_0,y_0}\bigl[(t\wedge\tau_0)+(t-\tau_0)^+\bigr]
	=H(x_0,y_0)+Ct.
\end{aligned}
\]
This proves \eqref{eq:cw-global-Foster}, which also gives, for \(T\ge1\),
\[
 \int_E W(u,v)\Pi_T^{(x_0,y_0)}(\dd u,\dd v)
 \le\frac{H(x_0,y_0)}{cT}+\frac Cc.
\]
Since \(\{(x,y)\in E:W(x,y)\le R\}\) is compact
for every \(R>0\), the preceding bound and Markov's
inequality imply tightness.
\end{proof}

For the SSP, we combine the occupation estimate
in \Cref{lem:predator-source-occupation} with tightness of the
terminal boundary distribution.

\begin{lemma}
\label{lem:abs-tightness}
Suppose that \Cref{ass:dissipative} holds.
Then, for every \((x_0,y_0)\in E\), the family
\(\{\Pi_T^{\tau_0,(x_0,y_0)}:T\ge1\}\) is tight.
If \((P_t^{\tau_0})_{t\ge0}\) is weak Feller, every weak limit
as \(T\to\infty\) is invariant and concentrated on \(\partial E\).
\end{lemma}

\begin{proof}
{ For \((x_0,y_0)\in\partial E\), the SSP satisfies
	\(Z_t^{\tau_0}=(x_0,y_0)\) for all \(t\ge0\), so $
	\Pi_T^{\tau_0,(x_0,y_0)}=\delta_{(x_0,y_0)}$ for $T>0.$ The conclusion is therefore immediate on \(\partial E\).
	It suffices to consider \((x_0,y_0)\in E^\circ\).}
For \(R>0\), let
\(K_R:=\{(u,v)\in E:W(u,v)\le R\}\), which is compact.
Define
\[
 Z_\partial:=
 \begin{cases}
 Z_{\tau_0},&\tau_0<\tau_\infty,\\
 (0,0),&\text{otherwise}.
 \end{cases}
\]
Then \(Z_\partial\) is an \(E\)-valued random variable, so its law is tight.
Given \(\varepsilon>0\), choose a compact \(K_\partial\subset E\)
such that
\(\Pp_{x_0,y_0}(Z_\partial\notin K_\partial)
<\varepsilon/2\).
Choose \(R>0\) such that
\((C+H(x_0,y_0))/(cR)<\varepsilon/2\), where \(C,c\) are the
constants in \eqref{eq:source-occupation}.
Observe that, almost surely,
\beqnn 
 \1_{\{Z_s^{\tau_0}\notin K_\partial\cup K_R\}}\le
 \1_{\{s<\tau_0\}}\1_{\{W(Z_s^{\tau_0})>R\}}
 +\1_{\{s\ge\tau_0\}}
  \1_{\{Z_\partial\notin K_\partial\}}.
 \eeqnn 
Observe that, for \(T\ge1\), \eqref{eq:source-occupation} now gives
\begin{align*}
 \Pi_T^{\tau_0,(x_0,y_0)}((K_\partial\cup K_R)^c)
 &\le\frac1T\E_{x_0,y_0}\int_0^{T\wedge\tau_0}
       \1_{\{W(Z_s)>R\}}\dd s
   +\Pp_{x_0,y_0}(Z_\partial\notin K_\partial)\\
 &\le\frac1{TR}\E_{x_0,y_0}\int_0^{T\wedge\tau_0}
       W(Z_s)\dd s+\frac{\varepsilon}{2}\\
 &\le\frac{C+H(x_0,y_0)}{cR}+\frac{\varepsilon}{2}
 <\varepsilon.
\end{align*}
Thus \(\{\Pi_T^{\tau_0,(x_0,y_0)}:T\ge1\}\) is tight.
Under the weak Feller assumption,
\Cref{lem:time-average-limits} shows that every weak limit
is invariant; it is concentrated on \(\partial E\) by
\Cref{prop:noimm-classification}. The result follows.
\end{proof}

We now prove \Cref{thm:noimm-empirical}.

\begin{proof}[Proof of \Cref{thm:noimm-empirical}]
{\textup{(i)}} By \Cref{lem:predator-source-occupation}, the SSP $Z^{\tau_0}$ is conservative. Part \textup{(i)} therefore follows
from \Cref{prop:noimm-classification}\textup{(i)}.

{\textup{(ii)}} Under \eqref{eq:boundary-compatible},
\Cref{lem:cw-occupation} gives conservativeness and tightness
for the CWP $Z$. Since \(b_{10},b_{20}>0\),
\eqref{eq:axis-nontrivial} holds. Then
\Cref{prop:noimm-classification}\textup{(ii)} shows that
\(\delta_{(0,0)}\) is its unique invariant probability.
If \((P_t)_{t\ge0}\) is weak Feller,
\Cref{lem:time-average-limits} gives
\(\Pi_T^{(x_0,y_0)}\Rightarrow\delta_{(0,0)}\) as
\(T\to\infty\), for every \((x_0,y_0)\in E\).
The result follows.
\end{proof}

\subsection{The competitive model}

We now prove \Cref{thm:competitive-noimm}, using a
first-moment estimate to establish conservativeness
and tightness of the time-averaged transition laws.

\begin{proof}[Proof of \Cref{thm:competitive-noimm}]
	Take \(c_1=1\) in $S$, which is given by \eqref{SVW}. By \eqref{eq:generator}, we have
	\[
	\cL S(x,y)=a_1x^{\theta_1}y^{\kappa_1}
	+a_2y^{\theta_2}x^{\kappa_2}
	-b_{10}x^{r_{10}}-b_{20}y^{r_{20}}\le0.
	\]
For \((x_0,y_0)\in E^\circ\), by \eqref{eq:localized-dynkin-inequality} and
Fatou's lemma, one sees that
\beqlb\label{eq3} 
 \E_{x_0,y_0}S(Z_{t\wedge\tau_N^+\wedge\tau_0})\le x_0+y_0.
\eeqlb
On \(\{\tau_N^+\le t\wedge\tau_0\}\), \(S(Z_{t\wedge\tau_N^+\wedge\tau_0})\ge N\). Hence
	\[
	\Pp_{x_0,y_0}(\tau_N^+\le t\wedge\tau_0)
	\le\frac{x_0+y_0}{N},
	\qquad t\ge0,\quad N>x_0\vee y_0.
	\]
	Letting \(N\to\infty\) excludes explosion before or at
	\(\tau_0\). Thus both realizations are conservative by
	Definition~\ref{def:absorbed} and
	Proposition~\ref{prop:axis-realization}.
	By \eqref{eq3} and Fatou's lemma, for every \((x_0,y_0)\in E\), one sees that
	$\E_{x_0,y_0}\bigl[X_{t\wedge\tau_0}+Y_{t\wedge\tau_0}\bigr]
		\le x_0+y_0$ and $\E_{x_0,y_0}\bigl[X_t+Y_t\bigr]
		\le x_0+y_0$ for $t \ge 0.$
	Consequently, both families of time-averaged transition
	probabilities are tight by Markov's inequality.
	The classification of invariant probabilities follows
	from \Cref{prop:noimm-classification}.
	Under \eqref{eq:axis-nontrivial}, the CWP $Z$ has the unique invariant probability
	\(\delta_{(0,0)}\).
	If its semigroup \((P_t)_{t\ge0}\) is also weak Feller, by
	\Cref{lem:time-average-limits}, $
	\Pi_T^{(x_0,y_0)}\Rightarrow\delta_{(0,0)}$ as $T\to\infty$ for any $(x_0,y_0)\in E.$
	The proof ends here.
\end{proof}

\subsection{The cooperative model}
\label{sec:cooperative-noimm-proofs}

 Throughout this subsection, fix a nonincreasing
function \(\chi\in C^\infty((0,\infty))\) such that
\(0\le\chi\le1\), \(\chi=1\) on \((0,1]\), and
\(\chi=0\) on \([2,\infty)\).  
We begin with the power calculation and the interior occupation
estimate used in the proofs below.

\begin{lemma}
\label{lem:cooperative-power-calculus}
For \(i=1,2\), \(u>0\), and \(q<1\),
\[
 \int_0^\infty\bigl[(u+z)^q-u^q-qu^{q-1}z\bigr]\mu_i(\dd z)
 =\frac{q(q-1)\Gamma(\alpha_i-q)}
        {\Gamma(\alpha_i)\Gamma(2-q)}\,u^{q-\alpha_i}.
\]
In particular, for \(p>0\),
\[
 \int_0^\infty\bigl[(u+z)^{-p}-u^{-p}+pu^{-p-1}z\bigr]
                 \mu_i(\dd z)
 =\frac{\Gamma(p+\alpha_i)}{\Gamma(p)\Gamma(\alpha_i)}
       u^{-p-\alpha_i}.
\]
\end{lemma}

\begin{proof}
{
With the change of variables \(z=uv\), the identities follow
from the power-integral formulas in \cite[Lemma~11]{XiongYangZhouCompeting}.
}
\end{proof}

{
The following lemma will be used to prove existence of an
invariant probability concentrated on \(E^\circ\).
}
\begin{lemma}
\label{lem:interior-lyapunov}
{Assume \eqref{eq:cooperative-regularity},
\(r_{10},r_{20}>1\), and \(\iota_1=\iota_2=0\).} Suppose that
\[
 V(x,y)=S(x,y)+G(x,y),
\]
where \(G\in C^2(E^\circ)\) is nonnegative and nonincreasing
in each variable and $S$ is given by \eqref{SVW}. Let \(W:E^\circ\to[0,\infty)\) be continuous.
Assume that, for every \(M>0\), both sets
\[
 \{(x,y)\in E^\circ:V(x,y)\le M\},\qquad
 \{(x,y)\in E^\circ:W(x,y)\le M\}
\]
are compact subsets of \(E^\circ\), and that
\begin{equation}
 \cL V\le C-cW\qquad\text{on }E^\circ
 \label{eq:interior-lyapunov-drift}
\end{equation}
holds for some \(C,c>0\). Then the CWP \(Z\) solving \eqref{eq:model}
with initial state \((x_0,y_0)\in E^\circ\) is global and remains in
\(E^\circ\) almost surely.
Its restricted semigroup is weak Feller, its time-averaged
transition laws are tight on \(E^\circ\), and every subsequential
limit is an invariant probability \(\pi\) satisfying
{\(\pi(E^\circ)=1\) and}
\(\int_{E^\circ}W\dd\pi\le C/c\).
\end{lemma}

\begin{proof}
Fix \((x_0,y_0)\in E^\circ\), and, for all sufficiently large
\(n\), set
\[
 \tau_n=\inf\{t\ge0:Z_t\notin(n^{-1},n)^2\},
 \qquad
 v_n=\inf_{(x,y)\in E^\circ\setminus(n^{-1},n)^2}V(x,y).
\]
For every \(M>0\), compactness places \(\{V\le M\}\) inside
\((n^{-1},n)^2\) for all sufficiently large \(n\); hence
\(v_n\ge M\) eventually and \(v_n\to\infty\) as $n \rightarrow \infty$.
Before \(\tau_n\), the Brownian and small-jump integrands have
bounded quadratic norms. For a positive jump of size \(z\)
in coordinate \(i\), nonnegativity and monotonicity of \(G\)
give
\[
 |V(x+z,y)-V(x,y)|\le z+G(x,y),\qquad
 |V(x,y+z)-V(x,y)|\le c_1z+G(x,y).
\]
Thus the large-jump increments and their compensators are
integrable before \(\tau_n\), since \(G\) and the jump rates
are bounded on \([n^{-1},n]^2\) and
\(\int_1^\infty z\,\mu_i(\dd z)<\infty\).
By \eqref{eq:interior-lyapunov-drift} and
\eqref{eq:localized-dynkin-inequality}, we obtain
\[
 \E_{x_0,y_0}V(Z_{t\wedge\tau_n})
 +c\E_{x_0,y_0}\int_0^{t\wedge\tau_n}W(Z_s)\dd s
 \le V(x_0,y_0)+Ct.
\]
Consequently,
\begin{equation}
 \Pp_{x_0,y_0}(\tau_n\le t)
 \le\frac{V(x_0,y_0)+Ct}{v_n}\rightarrow0
\label{eq:interior-exit-bound}
\end{equation}
as $n \rightarrow \infty$. Thus the CWP \(Z\) is global and remains in
	\(E^\circ\) almost surely, proving the first assertion. Hence by $\tau_n \uparrow \infty$ as $n \rightarrow \infty$ and Fatou's lemma, it follows that
\begin{equation*}
 \E_{x_0,y_0}V(Z_t)
 +c\E_{x_0,y_0}\int_0^tW(Z_s)\dd s
 \le V(x_0,y_0)+Ct.
\end{equation*}
In particular, the time-averaged law
\(\Pi_T^{(x_0,y_0)}\) satisfies
\begin{equation}
 \int_{E^\circ}W(x, y)\Pi_T^{(x_0,y_0)}(\dd x, \dd y)
 \le\frac{V(x_0,y_0)}{cT}+\frac Cc.
\label{eq:interior-time-average-bound}
\end{equation}
The compact sublevel sets of \(W\) give tightness
of \(\{\Pi_T^{(x_0,y_0)}:T\ge1\}\) on \(E^\circ\).
To prove the weak Feller property, let \(z_k\to z\) in
\(E^\circ\) and couple the CWP \(Z^{(k)}\) and the CWP \(Z\) with the same driving noises. Stop both when
either leaves \((n^{-1},n)^2\), at a time denoted by
\(\sigma_n^{(k)}\). The argument yielding
\eqref{eq:stopped-L1-continuity}, with additional stopping
when either solution reaches the lower boundary of
\((n^{-1},n)^2\), gives
\beqlb\label{eq4} 
 \E\bigl|Z^{(k)}_{t\wedge\sigma_n^{(k)}}
          -Z_{t\wedge\sigma_n^{(k)}}\bigr|_1
 \le e^{C_nt}|z_k-z|_1,
\eeqlb
where \(|(x,y)|_1:=|x|+|y|\).
For \(K=\{z,z_k:k\ge1\}\), \eqref{eq:interior-exit-bound} gives
\[
 \Pp(\sigma_n^{(k)}\le t)
 \le \frac{2\bigl(\sup_{w\in K}V(w)+Ct\bigr)}{v_n}.
\]
Together with \eqref{eq4} and
uniform continuity of $f$ on \([n^{-1},n]^2\), it implies
\(P_tf(z_k)\to P_tf(z)\) for \(f\in C_b(E^\circ)\),
first letting \(k\to\infty\) and then \(n\to\infty\).
Applying \Cref{lem:time-average-limits} with
\(D=E^\circ\) shows that every weak limit is invariant for
the common interior process, and hence for both boundary realizations.
As a probability on \(E^\circ\), its extension to \(E\)
satisfies \(\pi(E^\circ)=1\).
Taking \(T=T_k\to\infty\) along the convergent subsequence
in \eqref{eq:interior-time-average-bound} gives
\[
 \int_{E^\circ}W\,\dd\pi
 \le\liminf_{k\to\infty}
 \int_{E^\circ}W\,\dd\Pi_{T_k}^{(x_0,y_0)}
 \le C/c.
\]
The result follows.
\end{proof}

\begin{lemma}
\label{lem:cooperative-critical-extinction}
{Assume \eqref{eq:cooperative-regularity} and
\Cref{ass:cooperative-dissipation}, with
\(\iota_1=\iota_2=0\).}
If \eqref{eq:cooperative-critical-powers} holds, then, for every
\((x_0,y_0)\in E\), $
 Z_t\rightarrow(0,0)$ almost surely and in $L^1$ as 
$t\to\infty.$
\end{lemma}

\begin{proof}
The assumptions imply \(r_{10},r_{20}>1\), and hence
	\eqref{eq:imm-regularity} and
	\Cref{ass:weighted-dissipation}\textup{(iii)} hold.
	By \Cref{lem:imm-lyapunov}, the CWP \(Z\) is conservative. By \Cref{lem:parameter-drift}, choose \(c_1,c>0\)
so that \eqref{eq:weighted-dissipation} holds with \(C_0=0\).
Fix \((x_0,y_0)\in E\), and set
\(S_t:=S(Z_t)\), with \(S\) defined in \eqref{SVW}.
Since \(S=V_0-1\), \eqref{Vito} yields
\[
 S_t=S(x_0,y_0)+\int_0^t\cL S(Z_s)\,\dd s+M_t \quad \text{and} 
 \quad \cL S\le-cW_0\le0.
\]
Thus \((S_t)_{t\ge0}\) is a nonnegative supermartingale. As in the proof of \Cref{lem:imm-lyapunov},
stopping at \(\tau_N^+\) and letting \(N\to\infty\) give
\[
 \E_{x_0,y_0}S_t
 +c\E_{x_0,y_0}\int_0^t
   W_0(Z_s)\dd s
 \le S(x_0,y_0).
\]
Consequently, $ S_t\rightarrow S_\infty<\infty$ almost surely, and 
\begin{equation}
 \E_{x_0,y_0}\int_0^\infty
   W_0(Z_s)\dd s
 \le\frac{S(x_0,y_0)}{c}.
\label{eq:critical-extinction-integral}
\end{equation}
For every \(\varepsilon>0\),
\[
 \inf_{\substack{(x,y)\in E\\ S(x,y)\ge\varepsilon}}
       W_0(x,y)>0.
\]
Together with \eqref{eq:critical-extinction-integral},
this implies \(S_\infty=0\) almost surely and yields a sequence
\(t_n\uparrow\infty\) such that $\E_{x_0,y_0}
 W_0(Z_{t_n})
 \rightarrow0.$
By H\"older's inequality, 
\[
 \E_{x_0,y_0}S_{t_n}
 \le\bigl(\E_{x_0,y_0}(X_{t_n})^{r_{10}}\bigr)^{1/r_{10}}
   +c_1\bigl(\E_{x_0,y_0}(Y_{t_n})^{r_{20}}\bigr)^{1/r_{20}}
 \rightarrow0
\]
as $n \rightarrow \infty$.
Since \(t\mapsto\E_{x_0,y_0}S_t\) is nonincreasing,
\(\E_{x_0,y_0}S_t\to0\) as \(t\to\infty\). Together with \(S_t\to0\) almost
surely and \(x+y\le\max\{1,c_1^{-1}\}S(x,y)\), we obtain that $
 Z_t\rightarrow(0,0)$ almost surely and in $L^1$ as $t \rightarrow \infty$.
\end{proof}

\begin{proof}[Proof of \Cref{thm:cooperative-noimm-boundary}]
By \eqref{eq:cooperative-regularity} and
	\eqref{eq:cooperative-powers}, we have \(r_{i0}>1\)
	for \(i=1,2\). Hence \eqref{eq:imm-regularity} and
	\Cref{ass:weighted-dissipation}\textup{(iii)} hold.
	By \Cref{lem:imm-lyapunov}, the CWP \(Z\) is conservative. The SSP \(Z^{\tau_0}\)
is then conservative as well.
By \Cref{def:coordinatewise,def:absorbed}, \(\partial E\) is invariant
for both processes. We prove \(\pi(\partial E)=1\) in the three cases below.

Suppose first that \textup{(i)} holds. By
\Cref{lem:cooperative-critical-extinction}, the CWP converges to \((0,0)\) almost surely and in \(L^1\). For
\(f(x,y):=xy/[(1+x)(1+y)]\), by the convergence and the
definition of the SSP, as $t \rightarrow \infty$, we have
\[
 P_tf(x,y)\rightarrow0,
 \qquad P_t^{\tau_0}f(x,y)\rightarrow0,
 \qquad (x,y)\in E.
\]
By invariance and bounded convergence, every invariant probability
for either realization satisfies \(\pi(f)=0\), and hence
\(\pi(E^\circ)=0\).

Suppose next that \textup{(ii)} holds.
Since \(\theta_i\ge1\), \eqref{eq:cooperative-accessible-noise}
and \cite[Theorem~1.7]{XiongYangZhouEnhancing} give $
 \Pp_{x,y}(\tau_0<\infty)>0$ for $(x,y)\in E^\circ.$
The solution in \cite[Theorem~1.7]{XiongYangZhouEnhancing}, stopped at \(\tau_0\), agrees
with both realizations up to that time. Since neither process
can return to \(E^\circ\) after reaching \(\partial E\),
\[
 P_t\1_{E^\circ}(x,y)
 =P_t^{\tau_0}\1_{E^\circ}(x,y)
 =\1_{E^\circ}(x,y)\Pp_{x,y}(\tau_0>t).
\]
Thus every invariant probability \(\pi\) satisfies, for \(n\ge1\),
\[
 \pi(E^\circ)
 =\int_{E^\circ}\Pp_{x,y}(\tau_0>n)\,\pi(\dd x,\dd y),
 \qquad
 \int_{E^\circ}\Pp_{x,y}(\tau_0\le n)\,\pi(\dd x,\dd y)=0.
\]
Letting \(n\to\infty\),
\[
 \int_{E^\circ}\Pp_{x,y}(\tau_0<\infty)\,\pi(\dd x,\dd y)=0,
\]
and hence \(\pi(E^\circ)=0\).

Suppose now that \textup{(iii)} holds. Choose \(c_1,c,C_0\)
as in \Cref{lem:parameter-drift}, set $h=(\theta_1+\kappa_1)\wedge(\theta_2+\kappa_2)$ and 
\begin{equation*}
\beta_i:=\min\bigl\{r_{i1}-2\ \text{if }b_{i1}>0,\;
             r_{i2}-\alpha_i\ \text{if }b_{i2}>0\bigr\},
 \quad i=1,2,\qquad \beta:=\beta_1\vee\beta_2.
\end{equation*}
Fix \(p\in(0,1)\), and set
\begin{align*}
 H_p(x,y):={}&p(1-p)S^{p-2}
       \bigl(b_{11}x^{r_{11}}+c_1^2b_{21}y^{r_{21}}\bigr)\\
 &+\frac{p(1-p)\Gamma(\alpha_1-p)}
        {\Gamma(\alpha_1)\Gamma(2-p)}
       b_{12}x^{r_{12}}S^{p-\alpha_1}
 +\frac{p(1-p)\Gamma(\alpha_2-p)}
        {\Gamma(\alpha_2)\Gamma(2-p)}
       c_1^{\alpha_2}b_{22}y^{r_{22}}S^{p-\alpha_2},
\end{align*}
where $S$ is the function defined by \eqref{SVW}. By \eqref{eq:generator} and
\Cref{lem:cooperative-power-calculus},
\[
 \cL S^p=pS^{p-1}\cL S-H_p(x,y).
\]
 Since \(x\le S\) and \(y\le S/c_1\),
\(\cL S\le CS^h\) for \(0<S\le1\).
The definition of \(\beta_i\) gives constants \(d_1,d_2>0\) such that
\begin{align*}
 H_p(x,y)
 \ge d_1S^{p+\beta_1}\1_{\{x\ge S/2\}}
       +d_2S^{p+\beta_2}\1_{\{c_1y\ge S/2\}}\ge (d_1\wedge d_2)S^{p+\beta},
 \qquad 0<S\le1.
\end{align*}
 Thus, for \(c_*:=d_1\wedge d_2>0\),
\[
 \cL S^p\le CS^{p+h-1}-c_*S^{p+\beta},
 \qquad 0<S\le1.
\]
Since \(\beta<h-1\), choose \(0<\varepsilon\le1\) with
\(C\varepsilon^{h-1-\beta}<c_*\). Then
\(\cL S^p<0\) for \(0<S\le\varepsilon\).
For \(\lambda>0\), let $f_\lambda(s):=1-e^{-\lambda s^p}$ for $s \ge 0$. By \eqref{eq:generator}, one can check that
\[
 \cL f_\lambda(S)
 \le\lambda e^{-\lambda S^p}\cL S^p<0,
 \qquad 0<S\le\varepsilon.
\]
By \eqref{eq:weighted-dissipation}, \(\cL S<0\) for
\(S\ge R\), with some \(R>1\). Since \(f_\lambda\) is increasing and concave,
\[
 \cL f_\lambda(S)
 \le f_\lambda'(S)\cL S<0,\qquad S\ge R.
\]
It remains to consider \(\varepsilon\le S\le R\).
Notice that
\[
 \inf_{\substack{(x,y)\in E\\\varepsilon\le{S(x,y)}\le R}}
 \bigl(b_{11}x^{r_{11}}+c_1^2b_{21}y^{r_{21}}
       +b_{12}x^{r_{12}}+c_1^{\alpha_2}b_{22}y^{r_{22}}\bigr)>0.
\]
Set \(\delta:=p\min_{\varepsilon\le s\le R}s^{p-1}>0\).
Then $
 f_\lambda''(s)/f_\lambda'(s)
 =(p-1)/s-\lambda ps^{p-1}\le-\lambda \delta.$
Moreover,
\[
 0\le\frac{f_\lambda(s+z)-f_\lambda(s)}{f_\lambda'(s)}
 =\frac{1-e^{-\lambda[(s+z)^p-s^p]}}
 {\lambda ps^{p-1}}
 \le\frac1{\lambda \delta},
 \qquad \varepsilon\le s\le R,\quad z>0.
\]
Hence, for \(\lambda \delta>2\), it follows that
\begin{align*}
  \frac1{f_\lambda'(s)}
 \int_0^\infty\bigl[f_\lambda(s+z)-f_\lambda(s)
                    -zf_\lambda'(s)\bigr]\mu_i(\dd z) \le-\frac12\int_{2/(\lambda \delta)}^1z\,\mu_i(\dd z)
 \le-d_i\lambda^{\alpha_i-1}+C_i.
\end{align*}
Then one can check that $\cL f_\lambda(S) < 0$ for $\varepsilon\le S\le R$, and then $\cL f_\lambda <0$ on $E^\circ$ for large enough $\lambda$.

Let \((Q_t)_{t\ge0}\) be either \((P_t)_{t\ge0}\) or
\((P_t^{\tau_0})_{t\ge0}\), and let \(\pi\) be invariant
for \(Q_t\). Suppose that \(\pi(E^\circ)>0\).
Since \(\partial E\) is absorbing, stationarity gives 
\[
 \pi(\partial E)
 =\pi(\partial E)+\int_{E^\circ}Q_t(z,\partial E)\,\pi(\dd z).
\]
It follows that
\[
 \int_{E^\circ}Q_t(z,\partial E)\,\pi(\dd z)=0,
 \qquad
 \pi(A)=\int_{E^\circ}Q_t(z,A)\,\pi(\dd z),
 \quad A\subset E^\circ.
\]
Consequently, \(\pi_\circ(\cdot):=\pi(\cdot\cap E^\circ)/\pi(E^\circ)\)
is invariant for the process killed upon leaving \(E^\circ\). Started from
\(\pi_\circ\), the process remains in \(E^\circ\): indeed,
\(\Pp_{\pi_\circ}(\tau_0\le n)=0\) for every integer \(n\ge1\).
By \eqref{eq:localized-dynkin-inequality} and Fatou's lemma,
using \(0\le f_\lambda\le1\), for \(t>0\), we have
\[
\E_{\pi_\circ}f_\lambda(S(Z_t)) < \E_{\pi_\circ}f_\lambda(S(Z_t))
 +\E_{\pi_\circ}\int_0^t[-\cL(f_\lambda\circ S)](Z_s)\dd s
 \le\E_{\pi_\circ}f_\lambda(S(Z_0)),
\]
which is a contradiction.
Thus \(\pi(E^\circ)=0\), or equivalently, \(\pi(\partial E)=1\), in all three cases.
By \Cref{def:absorbed}, every probability concentrated on
\(\partial E\) is invariant for the SSP.
For the CWP, similar to the axis argument in the proof of
\Cref{prop:noimm-classification}\textup{(ii)}, with
\(b_{10},b_{20}>0\), it gives \(\pi=\delta_{(0,0)}\).
\end{proof}

\begin{proof}[Proof of \Cref{thm:cooperative-interior}]
For \(p,q>0\), consider \(\chi(x)x^{-p}\) and \(\chi(y)y^{-q}\), where $\chi(\cdot)$ is defined in the beginning of \Cref{sec:cooperative-noimm-proofs}.
Both functions are nonnegative and nonincreasing in each coordinate.
For \(0<x\le1\), \(\chi(x)x^{-p}=x^{-p}\) and
\(\chi(x+z)(x+z)^{-p}\le(x+z)^{-p}\).
On \(1\le x\le2\), the derivatives and L\'evy integral of
\(\chi(x)x^{-p}\) are bounded; for \(x\ge2\), they vanish.
The second coordinate is treated in the same way.
Thus by \Cref{lem:cooperative-power-calculus},
\begin{equation}
\begin{aligned}
 \cL\bigl(\chi(x)x^{-p}\bigr)\le{}&C+\1_{\{x\le1\}}p x^{-p}
 \biggl(b_{10}x^{r_{10}-1}+(p+1)b_{11}x^{r_{11}-2}\\
 &\hspace{3em}+\frac{\Gamma(p+\alpha_1)}
 {\Gamma(p+1)\Gamma(\alpha_1)}b_{12}x^{r_{12}-\alpha_1}
 -a_1x^{\theta_1-1}y^{\kappa_1}\biggr),\\
 \cL\bigl(\chi(y)y^{-q}\bigr)\le{}&C+\1_{\{y\le1\}}q y^{-q}
 \biggl(b_{20}y^{r_{20}-1}+(q+1)b_{21}y^{r_{21}-2}\\
 &\hspace{3em}+\frac{\Gamma(q+\alpha_2)}
 {\Gamma(q+1)\Gamma(\alpha_2)}b_{22}y^{r_{22}-\alpha_2}
 -a_2y^{\theta_2-1}x^{\kappa_2}\biggr)
\end{aligned}
\label{eq:cooperative-cutoff-generator}
\end{equation}
for some constant \(C>0\), uniform in \((x,y)\in E^\circ\).
The interaction terms omitted on the cutoff intervals are
nonpositive.

Suppose first that \textup{(i)} holds. Choose
\[
 \frac{\kappa_2}{r_2-\theta_2+1}<\rho<\frac{r_1-\theta_1+1}{\kappa_1},
 \qquad p=q\rho,
\]
where \(q>0\) will be taken sufficiently large. Put
\[
 A=p+1-\theta_1,\qquad B=q+1-\theta_2,
 \qquad
 \eta=\frac{AB-\kappa_1\kappa_2}{B+\kappa_1},\qquad
 \zeta=\frac{AB-\kappa_1\kappa_2}{A+\kappa_2}.
\]
For large \(q\), the above four quantities are positive. Moreover,
as \(q\to\infty\),
\begin{align*}
 p-\eta
 &=\theta_1-1+\kappa_1\frac{A+\kappa_2}{B+\kappa_1}
 \rightarrow\theta_1-1+\kappa_1\rho<r_1,\\
 q-\zeta
 &=\theta_2-1+\kappa_2\frac{B+\kappa_1}{A+\kappa_2}
 \rightarrow\theta_2-1+\frac{\kappa_2}{\rho}<r_2.
\end{align*}
Choose \(q\) large enough such that \(p-\eta<r_1\) and
\(q-\zeta<r_2\). By Young's inequality,
\[ {
	\frac{B}{B+\kappa_1}x^{-A}y^{\kappa_1}
	+\frac{\kappa_1}{B+\kappa_1}x^{\kappa_2}y^{-B}
	 \ge
	\bigl(x^{-A}y^{\kappa_1}\bigr)^{B/(B+\kappa_1)}
	\bigl(x^{\kappa_2}y^{-B}\bigr)^{\kappa_1/(B+\kappa_1)} 
	=x^{-\eta} }
 \]
and 
\beqnn
 \frac{\kappa_2}{A+\kappa_2}x^{-A}y^{\kappa_1}
 +\frac{A}{A+\kappa_2}x^{\kappa_2}y^{-B}
  \ge y^{-\zeta}.
\eeqnn
Consequently,
\[
 x^{-A}y^{\kappa_1}+x^{\kappa_2}y^{-B}
 \ge x^{-\eta}\vee y^{-\zeta}
 \ge\frac12(x^{-\eta}+y^{-\zeta}).
\]
Since \(A>\eta\) and \(B>\zeta\), separating the regions
\(x,y\le1\), \(x\le1<y\), \(y\le1<x\), and \(x,y>1\), we have
\begin{equation}
 \1_{\{x\le1\}}x^{-A}y^{\kappa_1}
 +\1_{\{y\le1\}}x^{\kappa_2}y^{-B}
 \ge\frac12(x^{-\eta}+y^{-\zeta})-1.
 \label{eq:cooperative-cutoff-coercivity}
\end{equation}
For \(\beta_1\ge r_1-p>-\eta\),
\(\beta_2\ge r_2-q>-\zeta\), and every \(\varepsilon>0\),
there is a constant \(C_\varepsilon>0\) such that
\begin{equation}
 \1_{\{x\le1\}}x^{\beta_1}
 \le\varepsilon x^{-\eta}+C_\varepsilon,
 \qquad
 \1_{\{y\le1\}}y^{\beta_2}
 \le\varepsilon y^{-\zeta}+C_\varepsilon.
 \label{eq:cooperative-power-absorption}
\end{equation}
 Combining
\eqref{eq:cooperative-cutoff-generator},
\eqref{eq:cooperative-cutoff-coercivity}, and
\eqref{eq:cooperative-power-absorption}, and choosing \(\varepsilon\) sufficiently small shows that
\begin{equation}
\cL\bigl(\chi(x)x^{-p}\bigr)+\cL\bigl(\chi(y)y^{-q}\bigr)
 \le C-\frac14((pa_1)\wedge (qa_2))(x^{-\eta}+y^{-\zeta})
 \qquad\text{on }E^\circ.
 \label{eq:cooperative-subcritical-drift}
\end{equation}
In this case, set \(\Psi(x,y)=\chi(x)x^{-p}+\chi(y)y^{-q}\).

Suppose now that \textup{(ii)} holds. For \(v>0\), let
\[
\begin{aligned}
 C_i(v):={}&b_{i0}\1_{\{r_{i0}-1=r_i\}}
 +(v+1)b_{i1}\1_{\{r_{i1}-2=r_i\}} +\frac{\Gamma(v+\alpha_i)}
 {\Gamma(v+1)\Gamma(\alpha_i)}
 b_{i2}\1_{\{r_{i2}-\alpha_i=r_i\}},
 \qquad i=1,2.
\end{aligned}
\]
Then \(C_i(r_i)=b_i\). Set \(\sigma=\kappa_2/(r_1-\theta_1+1)\).
There is a
constant \(R_0>0\) such that
\[
 \frac{b_1}{a_1}<R_0<
 \left(\frac{a_2}{b_2}\right)^{1/\sigma}.
\]
Choose \(\lambda_1,\lambda_2>0\) so that
\[
 \lambda_1r_1a_1
 =\sigma\lambda_2r_2a_2R_0^{-\sigma-1}.
\]
The function
\(R\mapsto\lambda_1r_1a_1R+
\lambda_2r_2a_2R^{-\sigma}\) attains its minimum at \(R_0\).
Consequently, for some \(c_0>0\),
\[
 \lambda_1r_1({b_1}-a_1R)
 +\lambda_2r_2({b_2}-a_2R^{-\sigma})\le-c_0,
 \qquad R>0.
\]
For \(\varepsilon\ge0\), let
\beqlb\label{eqpq} 
 p=r_1+\frac{r_1-\theta_1+1}{\kappa_1}\varepsilon,
 \qquad q=r_2+\varepsilon,
\eeqlb
and consider
\beqnn
 \frac{\lambda_1p[C_1(p)-a_1R]
 +\lambda_2qR^{-\varepsilon/\kappa_1}
       [C_2(q)-a_2R^{-\sigma}]}
 {1+R^{-\varepsilon/\kappa_1}}.
 \eeqnn
This quotient tends to \(-\infty\) at both endpoints of
\((0,\infty)\), uniformly for \(0\le\varepsilon\le1\),
and is at most \(-c_0/2\) when \(\varepsilon=0\).
Uniform continuity on compact intervals therefore gives
a sufficiently small \(\varepsilon\in(0,1)\) and
\(c_2>0\) such that
\begin{equation}
\begin{aligned}
\lambda_1p[C_1(p)-a_1R]
 +\lambda_2qR^{-\varepsilon/\kappa_1}
       [C_2(q)-a_2R^{-\sigma}] \le-c_2(1+R^{-\varepsilon/\kappa_1}),
 \qquad R>0.
\end{aligned}
\label{eq:cooperative-critical-scalar}
\end{equation}
Set
\[
 \eta=\frac{r_1-\theta_1+1}{\kappa_1}\varepsilon,
 \qquad
 \Psi(x,y)=\lambda_1\chi(x)x^{-p}+\lambda_2\chi(y)y^{-q}.
\]
With \(R=y^{\kappa_1}/x^{r_1-\theta_1+1}\), by the assumption
\((r_1-\theta_1+1)(r_2-\theta_2+1)=\kappa_1\kappa_2\), we have
\[
 y^{-\varepsilon}=x^{-\eta}R^{-\varepsilon/\kappa_1},
 \qquad
 \frac{x^{\kappa_2}}{y^{r_2-\theta_2+1}}=R^{-\sigma}.
\]
For \(0<x,y\le1\), \eqref{eq:cooperative-cutoff-generator}
and \eqref{eqpq} yield
\begin{align*}
 \cL\Psi(x,y)\le{}&C+
 \lambda_1p x^{-\eta}[C_1(p)-a_1R+o(1)]+\lambda_2q y^{-\varepsilon}[C_2(q)-a_2R^{-\sigma}+o(1)]\\
 \le{}&C-c_2(x^{-\eta}+y^{-\varepsilon})
       +o(1)x^{-\eta}+o(1)y^{-\varepsilon},
\end{align*}
where the second inequality uses
\eqref{eq:cooperative-critical-scalar}.
Hence, for some \(r\in(0,1)\),
\[
 \cL\Psi(x,y)\le C-\frac{c_2}{2}
 (x^{-\eta}+y^{-\varepsilon}),\qquad 0<x,y\le r.
\]
If \(x\le r<y\), \eqref{eq:cooperative-cutoff-generator}
instead gives, for suitable constants \(C,c_3,c_4>0\),
\[
 \cL\Psi(x,y)
 \le-c_3x^{-\eta-r_1+\theta_1-1}+Cx^{-\eta}+C
 \le C-c_4x^{-\eta}.
\]
Interchanging the
coordinates treats \(y\le r<x\). On \(x,y\ge r\), the
positive terms and the two negative powers on the right hands of \eqref{eq:cooperative-cutoff-generator} are bounded.
Thus, after adjusting the constants,
\begin{equation}
 \cL\Psi(x,y)\le C-c_5(x^{-\eta}+y^{-\varepsilon})
 \qquad\text{on }E^\circ.
 \label{eq:cooperative-critical-drift}
\end{equation}

In either case, choose \(c_1,c,C_0\) as in
\Cref{lem:parameter-drift}, and define $
 V(x,y):=S(x,y)+\Psi(x,y).$
By \eqref{eq:weighted-dissipation},
\eqref{eq:cooperative-subcritical-drift}, and
\eqref{eq:cooperative-critical-drift}, there are constants
\(\widetilde C,\widetilde c>0\) such that
\[
 \cL V(x,y)\le
 \begin{cases}
 \widetilde C-\widetilde c\bigl(W_0(x,y)+x^{-\eta}+y^{-\zeta}\bigr),
       &\text{in case \textup{(i)}},\\
 \widetilde C-\widetilde c\bigl(W_0(x,y)+x^{-\eta}+y^{-\varepsilon}\bigr),
       &\text{in case \textup{(ii)}}.
 \end{cases}
\]
The result follows by \Cref{lem:interior-lyapunov}.
\end{proof}

We first establish two lemmas for the critical case \(\beta=1\).

\begin{lemma}\label{lem:sharp-boundary-lyapunov}
	Assume the hypotheses of \Cref{thm:cooperative-sharp} hold with
	\(\beta=1\) and \(\widehat q_1+\widehat q_2>1\).
	Then, for \(S,W_0\) defined in \eqref{SVW} with \(c_1=1\),
	there exists a nonnegative, coordinatewise nonincreasing
	\(F\in C^2(E^\circ)\) such that \(V=S+F\) and \(W=W_0+SF\)
	have compact sublevel sets in \(E^\circ\) and satisfy
	\(\cL V\le C-cW\) for some \(C,c>0\).
\end{lemma}

\begin{proof}
	By \eqref{q1q2}, we have $
	\Psi_1(\widehat q_1)<a_2\widehat q_2$ and
	$\Psi_2(\widehat q_2)<a_1\widehat q_1.$
	Choose $
	0<\delta<\frac14\min\bigl\{
	a_2\widehat q_2-\Psi_1(\widehat q_1),\,
	a_1\widehat q_1-\Psi_2(\widehat q_2)\bigr\}.$
	Using the cutoff \(\chi\)
	fixed at the beginning of
		\Cref{sec:cooperative-noimm-proofs}, we set
	\[
	g_i(u)=1+\varepsilon_i\chi(u)u^{-\widehat q_i},
	\qquad F(x,y)=g_1(x)g_2(y),
	\]
	where \(\varepsilon_i>0\) will be chosen below. Let
	\beqlb\label{eqmcA}
	\mathcal A_i g(u):= -b_{i0}u^{r_{i0}}g'(u)
	+b_{i1}u^3g''(u) +b_{i2}u^{\alpha_i+1}
	\int_0^\infty
	[g(u+z)-g(u)-zg'(u)]\,\mu_i(\dd z).
	\eeqlb
	For \(0<u\le1\), \Cref{lem:cooperative-power-calculus} and
	\(g_i(u+z)\le1+\varepsilon_i(u+z)^{-\widehat q_i}\)
	give {
		\[
		\begin{aligned}
			\frac{\mathcal A_i g_i(u)}{g_i(u)}
			\le\frac{\varepsilon_i\Psi_i(\widehat q_i)u
				+\widehat q_i b_{i0}\varepsilon_i u^{r_{i0}-1}}
			{u^{\widehat q_i}+\varepsilon_i}
			\le\Psi_i(\widehat q_i)u
			+\frac{\widehat q_i b_{i0}\varepsilon_i u^{r_{i0}-1}}
			{u^{\widehat q_i}+\varepsilon_i}.
		\end{aligned}
		\]
		Since \(r_{i0}>2\),
		\[
		\lim_{\varepsilon_i\downarrow0}\sup_{0<u\le1}
		\frac{\varepsilon_i u^{r_{i0}-2}}
		{u^{\widehat q_i}+\varepsilon_i}=0.
		\]
		On \([1,2]\), the three terms on the right hand of
		\eqref{eqmcA} are \(O(\varepsilon_i)\).
		For \(u\ge2\), all those terms vanish. We may therefore fix
		\(\varepsilon_1,\varepsilon_2>0\) sufficiently small that
		\[
		\frac{\mathcal A_i g_i(u)}{g_i(u)}
		\le[\Psi_i(\widehat q_i)+\delta]u,
		\qquad u>0,\quad i=1,2
		\]
		for constant $
		0<\delta<\frac14\min\{
		a_2\widehat q_2-\Psi_1(\widehat q_1),\,
		a_1\widehat q_1-\Psi_2(\widehat q_2)\}.$
		Let $
		h_i(u)=-\frac{ug_i'(u)}{g_i(u)}\ge0.$
		Then \(h_i(u)\to\widehat q_i\) as \(u\downarrow0\), and
		\eqref{eq:generator} gives
		\begin{align*}
			\frac{\cL F(x,y)}{F(x,y)}
			&=\frac{\mathcal A_1g_1(x)}{g_1(x)}
			+\frac{\mathcal A_2g_2(y)}{g_2(y)}
			-a_2xh_2(y)-a_1yh_1(x)\\
			&\le[\Psi_1(\widehat q_1)+\delta-a_2h_2(y)]x
			+[\Psi_2(\widehat q_2)+\delta-a_1h_1(x)]y.
		\end{align*}
		Choose \(r\in(0,1)\) so that the first bracket is at most
		\(-\tilde\delta\) for all \(y\le r\) and some
		$\tilde\delta>0$, and the second is at most
		\(-\tilde\delta\) for all \(x\le r\). Next choose
		\[
		0<\eta\le r\min\left\{1,\,
		\frac{\tilde\delta/2}{\Psi_1(\widehat q_1)+\delta},\,
		\frac{\tilde\delta/2}{\Psi_2(\widehat q_2)+\delta}\right\}.
		\]
		If \(x,y\le r\), then
		$\cL F(x,y)/F(x,y)$ is at most
		\(-\tilde\delta(x+y)\).
		If \(x\ge r\) and \(y\le\eta\), we have
		\[
		\frac{\cL F(x,y)}{F(x,y)}
		\le-\tilde\delta x+
		[\Psi_2(\widehat q_2)+ \delta]y
		\le-\frac{\tilde\delta x}{2}
		\le-\frac{\tilde\delta}{4}(x+y).
		\]
		Interchanging the coordinates treats
		\(y\ge r\) and \(x\le\eta\).
		Hence $
		\cL F\le-\frac{\tilde\delta}{4}SF$
		when $x\wedge y\le\eta.$
		On \(\{x\wedge y>\eta\}\), \(F\) is bounded and
		\(\cL F\le C_\eta S\), which yields $
		\cL F\le-\frac{\tilde\delta}{4}SF+C_\eta S$
		on \(E^\circ\).
		By \Cref{lem:parameter-drift},
		\(\cL S\le C_0-c_0W_0\) with \(c_1=1\).
		Thus
		\(\cL V\le C-cW\).
		At infinity, \(S\) controls \(V\) and \(W_0\) controls \(W\).
		Away from the origin, both \(F\) and \(SF\) diverge on
		approach to either axis. As $(x,y)\rightarrow(0,0)$,
		\[
		SF\ge\varepsilon_1\varepsilon_2
		(x+y)x^{-\widehat q_1}y^{-\widehat q_2}
		\ge\varepsilon_1\varepsilon_2
		(x+y)^{1-\widehat q_1-\widehat q_2}\rightarrow\infty.
		\]
		Also $\lim_{x,y\rightarrow0}F(x,y)=\infty$.}
	Thus \(V\) and \(W\) have compact sublevel sets
	in \(E^\circ\).
\end{proof}

\begin{lemma}\label{lem:sharp-stationary-tests}
	Assume the hypotheses of \Cref{thm:cooperative-sharp}
	with \(\beta=1\). Let \(\pi\) be an invariant probability
	{ for either the CWP \(Z\) or the SSP \(Z^{\tau_0}\),}
	with \(\pi(E^\circ)=1\). Then {$\int_{E^\circ}(x/y+y/x)
	 \pi(\dd x, \dd y)<\infty.$} 
	Moreover, for every \(p,q>0\) with \(p+q=1\), the function
	\(f(x,y)=x^{-p}y^{-q}\) satisfies $
	\cL f\in L^1(\pi)$ and $\int_{E^\circ}\cL f\,\dd\pi=0.$
\end{lemma}

\begin{proof}
{
	Since \(\partial E\) is invariant for both the CWP \(Z\)
	and the SSP \(Z^{\tau_0}\), by stationarity and
	\(\pi(E^\circ)=1\), we have $
	\Pp_\pi(\tau_0\le t)
	\le\Pp_\pi(Z_t\in\partial E)=\pi(\partial E)=0$ for $t \ge 0$.
	Thus \(\Pp_\pi(\tau_0=\infty)=1\), and the two processes
	coincide under \(\Pp_\pi\).
	By \Cref{lem:common-stationary-moments}, $\int_{E^\circ}W_0(x,y)\pi(\dd x,\dd y)<\infty,$
	where \(W_0\) is defined in \eqref{SVW}.}
	Choose
	\[
	\phi(s)=
	\begin{cases}
		s,&0\le s\le1,\\
		\psi(s),&1<s<2,\\
		\psi(2),&s\ge2,
	\end{cases}
	\]
	where \(\psi\) is a smooth interpolation chosen so that
	\(\phi\) is smooth, nondecreasing and concave on
	\([0,\infty)\), and let \(\phi_n(s)=n\phi(s/n)\).
	For \(s>0\) and \(0\le d\le s\), one can check that $|\phi_n(s-d)-\phi_n(s)|^2
	\le Cn\frac{d^2}{s}$ and
	\beqlb\label{eq1007}
		0\le-[\phi_n(s-d)-\phi_n(s)+\phi_n'(s)d]
		\le C\frac{d^2}{s}\1_{\{s\ge n\}} 
	\eeqlb
	with \(C\) depending only on \(\phi\).
	{Indeed, \(0\le\phi_n'\le1\), \(\|\phi_n''\|_\infty\le C/n\) and	\(\|\phi_n\|_\infty\le Cn\). If \(s<n\), the Taylor	remainder vanishes because \(\phi_n\) is the identity	on \([0,n]\). For \(n\le s\le4n\), by Taylor's formula,
	\[	0\le-[\phi_n(s-d)-\phi_n(s)+\phi_n'(s)d]\le C\frac{d^2}{n}\le4C\frac{d^2}{s}	\]
	and
	\[	|\phi_n(s-d)-\phi_n(s)|^2\le d^2\le4n\frac{d^2}{s}.	\] 
 If \(s>4n\), then \(\phi_n'(s)=0\).
	When \(s-d\ge2n\), both \(s-d\) and \(s\) lie in the
	region where \(\phi_n\) is constant, so both expressions
	vanish.
	Otherwise, \(s-d<2n\), and hence $
	d>s-2n>\frac{s}{2}$, $\frac{d^2}{s}>\frac{s}{4}>n.$
	Using \(\|\phi_n\|_\infty\le Cn\), we obtain
	\[
	\begin{aligned}
		0\le
		-[\phi_n(s-d)-\phi_n(s)+\phi_n'(s)d]
		&=\phi_n(s)-\phi_n(s-d)
		\le Cn\le C\frac{d^2}{s},\\
		|\phi_n(s-d)-\phi_n(s)|^2
		&\le Cn^2\le Cn\frac{d^2}{s}.
	\end{aligned}
	\]
	This proves the two estimates for \(s>4n\).}

	First take \(R(x,y)=x^{-1}+y^{-1}\). By
	\Cref{lem:cooperative-power-calculus},
{	
	\begin{equation}
		\cL R(x,y)=\Psi_1(1)+\Psi_2(1)
		+b_{10}x^{r_{10}-2}+b_{20}y^{r_{20}-2}
		-a_1\frac{y}{x}-a_2\frac{x}{y}.
		\label{eq:sharp-R-generator}
	\end{equation}
  Writing \(R=R(x,y)\) for simplification,
	by the concavity of \(\phi_n\), \eqref{eq:generator}, \eqref{eq1007} and \eqref{eq:sharp-R-generator}, we have
	\[
	\begin{aligned}
		0&\le\phi_n'(R)\cL R-\cL(\phi_n\circ R)\\
		&\le-\phi_n''(R)
		\left(\frac{b_{11}}x+\frac{b_{21}}y\right)
		+\frac{C}{R}\1_{\{R\ge n\}}
		\Bigg[
		b_{12}x^{\alpha_1+1}
		\int_0^\infty \left(\frac1x-\frac1{x+z}\right)^2\,\mu_1(\dd z)\cr 
		&\qquad\qquad\qquad\qquad\qquad\qquad\qquad\qquad  +b_{22}y^{\alpha_2+1}
		\int_0^\infty \left(\frac1y-\frac1{y+z}\right)^2\,\mu_2(\dd z)
		\Bigg].
	\end{aligned}
	\]
	The substitutions \(z=xt\), respectively, yield
	\[
	x^{\alpha_i+1}\int_0^\infty \left(\frac1x-\frac1{x+z}\right)^2\,\mu_i(\dd z)
	=\frac{\alpha_i(\alpha_i-1)}x.
	\]
	Moreover, since \(\|\phi_n''\|_\infty\le C/n\) and
	\(\phi_n''=0\) on \([n,2n]^c\),
	\[
	0\le-\phi_n''(R)
	\left(\frac{b_{11}}x+\frac{b_{21}}y\right)
	\le\frac{CR}{n}\1_{\{n\le R\le2n\}}
	\le C\1_{\{R\ge n\}}.
	\]
	Combining these estimates with \(x^{-1}+y^{-1}=R\), we obtain
	\beqlb\label{eq1006b}
	0\le\phi_n'(R(x,y))\cL R(x,y)
	-\cL(\phi_n\circ R)(x,y)
	\le C\1_{\{R(x,y)\ge n\}}.
	\eeqlb
	For fixed \(n\), by \(\phi_n'(R(x,y))R(x,y)\le2n\), we have $
	\phi_n'(R(x,y))\frac{y}{x}\le2ny$ and $
	\phi_n'(R(x,y))\frac{x}{y}\le2nx.$
	Using \(r_{10},r_{20}>2\) and
	\eqref{eq:sharp-R-generator}, we obtain
	\[
	|\phi_n'(R(x,y))\cL R(x,y)|
	\le C_n(1+W_0(x,y)).
	\]
	Hence by the above and \eqref{eq1006b},
	\beqnn
	|\cL (\phi_n\circ R(x, y))| \le |\phi_n'(R(x, y))\cL R(x, y)| + C\1_{\{R(x,y)\ge n\}} \le C_n(1+W_0(x,y)) + C.
	\eeqnn
It follows that \(\cL(\phi_n\circ R)\in L^1(\pi)\).
	As in \eqref{eq:localized-dynkin-martingale}, by It\^o's
	formula,
	\beqlb\label{Mt} 
	M_t^{n,R}:=\phi_n(R(Z_t))-\phi_n(R(Z_0))
	-\int_0^t\cL(\phi_n\circ R)(Z_s)\dd s 
	\eeqlb 
	is a local martingale. 
	For every \(T>0\), 
	\[
	\sup_{0\le t\le T}|M_t^{n,R}|
	\le2\|\phi_n\|_\infty
	+\int_0^T|\cL(\phi_n\circ R)(Z_s)|\dd s.
	\]
	Under \(\Pp_\pi\), \(Z_s\) has distribution \(\pi\) for every
	\(s\ge0\). Hence,  
	\[
	\begin{aligned}
		\E_\pi\int_0^T
		|\cL(\phi_n\circ R)(Z_s)|\dd s
	  =T\int_{E^\circ}
		|\cL(\phi_n\circ R)(z)|\,\pi(\dd z)
		<\infty.
	\end{aligned}
	\]
  Thus localization and dominated
	convergence show that \(M^{n,R}\) is a martingale.
	Taking stationary expectations to \eqref{Mt}, we have
	\begin{equation}
		0=\E_\pi[\phi_n(R(Z_t))-\phi_n(R(Z_0))]
		=t\int_{E^\circ}\cL(\phi_n\circ R)(x,y)
		\pi(\dd x,\dd y),\qquad t>0.
		\label{eq:sharp-R-stationarity}
	\end{equation}
	Consequently, by \eqref{eq:sharp-R-generator},
	\eqref{eq:sharp-R-stationarity}, and the nonnegative
	generator difference above,
	\[
	\begin{aligned}
		\int_{E^\circ}\phi_n'(R(x,y))
		\left(a_1\frac{y}{x}+a_2\frac{x}{y}\right)\pi(\dd x,\dd y)
		&\le\Psi_1(1)+\Psi_2(1)\\
		&\quad+\int_{E^\circ}
		(b_{10}x^{r_{10}-2}+b_{20}y^{r_{20}-2})
		\pi(\dd x,\dd y).
	\end{aligned}
	\] }
	Letting \(n\to\infty\), the first assertion follows
	from Fatou's lemma.
	
	Now fix \(p,q>0\) with \(p+q=1\), and set
		\(f(x,y)=x^{-p}y^{-q}\). By \eqref{eq:generator},
	\[
	\frac{\cL f(x,y)}{f(x,y)}
	=[\Psi_1(p)-a_2q]x+[\Psi_2(q)-a_1p]y
	+p b_{10}x^{r_{10}-1}+q b_{20}y^{r_{20}-1}.
	\]
	Since \(xf(x,y)=(x/y)^q\) and \(yf(x,y)=(y/x)^p\),
	the first assertion gives $
	\int_{E^\circ}(xf(x,y)+yf(x,y))\pi(\dd x,\dd y)<\infty.$ 
Similar to \eqref{eq1006b}, one can check that
\[
0\le \phi_n'(f)\cL f-\cL(\phi_n\circ f)
\le C(xf+yf)\1_{\{f\ge n\}} 
\]
and \(\cL(\phi_n\circ f)\in L^1(\pi)\).
Repeating the argument leading to
\eqref{eq:sharp-R-stationarity}, we obtain $
\int_{E^\circ}\cL(\phi_n\circ f)(x, y)\,\pi(\dd x, \dd y)=0.$
Consequently,
\beqlb
	\label{eq:sharp-truncated-f-identity}
		0
		\ar\le\ar \int_{E^\circ}\phi_n'(f(x, y))\cL f(x, y)\,\pi(\dd x, \dd y) =\int_{E^\circ}
		[\phi_n'(f(x, y))\cL f(x, y)-\cL(\phi_n\circ f)(x, y)]\,\pi(\dd x, \dd y)\cr
		\ar\le\ar C\int_{E^\circ}(xf(x, y)+yf(x, y))
		\1_{\{f(x, y)\ge n\}}\,\pi(\dd x, \dd y).
	\eeqlb
	The right-hand side of \eqref{eq:sharp-truncated-f-identity} tends to zero as \(n\to\infty\)
	by dominated convergence.
	The nonnegative self-regulation terms in \(\cL f\) imply $
	(\cL f)^-(x,y)\le C(xf(x,y)+yf(x,y)),$
	so \((\cL f)^-\in L^1(\pi)\). Hence, as \(n\to\infty\),
	\[
	\begin{aligned}
		\int_{E^\circ}\phi_n'(f(x,y))(\cL f)^+(x,y)
		\pi(\dd x,\dd y)
		&=\int_{E^\circ}\phi_n'(f(x,y))\cL f(x,y)
		\pi(\dd x,\dd y)\\
		&\quad+\int_{E^\circ}\phi_n'(f(x,y))(\cL f)^-(x,y)
		\pi(\dd x,\dd y)\\
		&\longrightarrow\int_{E^\circ}(\cL f)^-(x,y)
		\pi(\dd x,\dd y).
	\end{aligned}
	\]
	Fatou's lemma gives \((\cL f)^+\in L^1(\pi)\), and hence
	\(\cL f\in L^1(\pi)\). Letting \(n\to\infty\) in
	\eqref{eq:sharp-truncated-f-identity} and applying dominated
	convergence, we obtain $
	\int_{E^\circ}\cL f(x,y)\pi(\dd x,\dd y)=0.$
	The result follows.
\end{proof}

\begin{proof}[Proof of \Cref{thm:cooperative-sharp}]
		Since \(\theta_i=\kappa_i=1\) and \(r_{10},r_{20}>2\),
	both inequalities in \eqref{eq:cooperative-powers} are
	strict. Thus \Cref{ass:cooperative-dissipation} holds. By \Cref{lem:imm-lyapunov}, the CWP \(Z\) is conservative,
	and hence so is the SSP \(Z^{\tau_0}\).
	If \(\beta<1\), the left-hand side of condition
	\textup{(iii)} in \Cref{thm:cooperative-noimm-boundary}
	equals \(\beta\), whereas its right-hand side equals \(1\).
	Hence the invariant probabilities
	of the SSP \(Z^{\tau_0}\) are exactly those concentrated on
	\(\partial E\), and \(\delta_{(0,0)}\) is the unique
	invariant probability of the CWP \(Z\).
	If \(\beta>1\), then $
	r_i=(r_{i0}-1)\wedge\beta>1$ for $i = 1, 2$. 
	It follows that $
	(r_1-\theta_1+1)(r_2-\theta_2+1)
	=r_1r_2>1=\kappa_1\kappa_2.$
	By \Cref{thm:cooperative-interior}\textup{(i)}, the CWP \(Z\)
	and the SSP \(Z^{\tau_0}\) have a common nondegenerate
	invariant probability concentrated on \(E^\circ\).
	
	It remains to consider \(\beta=1\). Each \(\Psi_i\)
	extends continuously to \(\Psi_i(0)=0\) and is strictly
	increasing. Hence the defining equations have unique
	solutions \(\widehat q_1,\widehat q_2\in(0,1)\).
	If \(\widehat q_1+\widehat q_2>1\),
	\Cref{lem:sharp-boundary-lyapunov,lem:interior-lyapunov}
	give an invariant probability concentrated on
	\(E^\circ\), common to the CWP \(Z\)
		and the SSP \(Z^{\tau_0}\).
	
	{
	Suppose that \(\widehat q_1+\widehat q_2\le1\), and that
	an invariant probability \(\pi\) for the CWP \(Z\)
	or the SSP \(Z^{\tau_0}\) satisfies \(\pi(E^\circ)>0\).
	Notice that \(\partial E\)
	is invariant for the CWP and the SSP.
	As in the proof of \Cref{thm:cooperative-noimm-boundary}
	\textup{(iii)}, the normalized restriction
	\(\pi(\cdot\cap E^\circ)/\pi(E^\circ)\) is also invariant.
	Replacing \(\pi\) by this restriction, assume
	\(\pi(E^\circ)=1\). Choose
	\(p\in[\widehat q_1,1-\widehat q_2]\) and  \(q=1-p\).
	Then by \(p,q\in(0,1)\), we have $\Psi_1(p) \ge\Psi_1(\widehat q_1)
	=a_2(1-\widehat q_1)\ge a_2q$ and $\Psi_2(q) \ge a_1p.$
  For \(f(x,y)=x^{-p}y^{-q}\)},
	\[
	\frac{\cL f(x,y)}{f(x,y)}
	=[\Psi_1(p)-a_2q]x+[\Psi_2(q)-a_1p]y
	+p b_{10}x^{r_{10}-1}+q b_{20}y^{r_{20}-1}>0.
	\]
	This contradicts
	\(\int_{E^\circ}\cL f(x,y)\pi(\dd x,\dd y)=0\)
	from \Cref{lem:sharp-stationary-tests}.
	Thus every invariant probability is concentrated on
	\(\partial E\).
	By \Cref{def:absorbed}, every probability concentrated on
	\(\partial E\) is invariant for the SSP \(Z^{\tau_0}\).
	For the CWP \(Z\), the axis argument in the proof of
	\Cref{prop:noimm-classification}\textup{(ii)}, using
	\(b_{10},b_{20}>0\), shows that \(\delta_{(0,0)}\) is
	the unique invariant probability.
\end{proof}

\section{Models with immigration}
\label{sec:imm-proofs}

Throughout this section, we assume \eqref{eq:imm-regularity} and
\Cref{ass:weighted-dissipation} hold. Fix \(c_1,c,C_0\) as in
\Cref{lem:parameter-drift}, so that
\eqref{eq:weighted-dissipation} holds. Here
\(\cL\) denotes the operator in \eqref{eq:generator}.

\subsection{{Proof of  \Cref{thm:imm-existence}}}

We first prove the weak Feller property.

\begin{lemma} 
\label{lem:imm-weak-feller}
The transition semigroup \((P_t)_{t\ge0}\) is weak Feller.
\end{lemma}

\begin{proof}
Fix \(t>0\) and let \((x_n,y_n)\to(x_0,y_0)\) in \(E\).
On one probability space, use \(B_1,B_2,N_1,N_2\) to construct
the solutions \((X^{(n)},Y^{(n)})\) and \(Z\) with these
initial states. For \(R>1\), let
\[
 \rho_R^{(n)}:=\inf\{s\ge0:
 X_s^{(n)}\vee Y_s^{(n)}\vee X_s\vee Y_s\ge R\}.
\]
By convergence of the initial states,
\Cref{lem:imm-lyapunov} gives
\begin{equation}
 \lim_{R\to\infty}\sup_{n\ge1}
 \Pp(\rho_R^{(n)}\le t)=0.
 \label{eq:imm-uniform-exit}
\end{equation}
For fixed \(R\), \eqref{eq:stopped-L1-continuity} in
\Cref{prop:imm-wellposed} yields
\[
 \E\bigl[|X_{t\wedge\rho_R^{(n)}}^{(n)}-X_{t\wedge\rho_R^{(n)}}|
 +|Y_{t\wedge\rho_R^{(n)}}^{(n)}-Y_{t\wedge\rho_R^{(n)}}|\bigr]
 \le\e^{C_Rt}(|x_n-x_0|+|y_n-y_0|),
\]
{ which gives}
\[
 |X_{t\wedge\rho_R^{(n)}}^{(n)}-X_{t\wedge\rho_R^{(n)}}|
 +|Y_{t\wedge\rho_R^{(n)}}^{(n)}-Y_{t\wedge\rho_R^{(n)}}|
 \rightarrow0
\]
{in probability as \(n\to\infty\).
For \(f\in C_b(E)\), we have}
\[
 \E\left[
 \left|f(X_t^{(n)},Y_t^{(n)})-f(Z_t)\right|
 \1_{\{\rho_R^{(n)}>t\}}
 \right]\rightarrow0\quad{\text{as }n\to\infty}.
\]
Consequently,
\begin{align*}
 &\left|P_tf(x_n,y_n)-P_tf({x_0,y_0})\right|\\
 &\quad\le
 \E\left[
 \left|f(X_t^{(n)},Y_t^{(n)})-f(Z_t)\right|
 \1_{\{\rho_R^{(n)}>t\}}
 \right]
 +2\lVert f\rVert_\infty\Pp(\rho_R^{(n)}\le t).
\end{align*}
Taking first \(n\to\infty\) and then \(R\to\infty\),
and using \eqref{eq:imm-uniform-exit}, we obtain
\(P_tf(x_n,y_n)\to P_tf(x_0,y_0)\).
Thus \(P_tf\in C_b(E)\). The result follows.
\end{proof}

The preceding estimates give the following compactness result.

\begin{lemma}
\label{lem:imm-time-averages}
For every \((x_0,y_0)\in E\), the family
\(\{\Pi_T^{(x_0,y_0)}:T\ge1\}\) is tight, and every weak
limit as \(T\to\infty\) is an invariant probability.
\end{lemma}

\begin{proof}
Nonexplosion and conservativeness follow from
\Cref{lem:imm-lyapunov}, while \Cref{lem:imm-weak-feller} gives the
weak Feller property.  {By \Cref{lem:imm-lyapunov} again,
there exists a constant \(C>0\) such that, for \(T\ge1\),}
\[
 \int_E {V_0}(u,v)\,\Pi_T^{({x_0,y_0})}(\dd u,\dd v)
 =\frac1T\int_0^T\E_{{x_0,y_0}}{V_0}(Z_t)\dd t
 \le\frac{{V_0}({x_0,y_0})}{T}+C.
\]
Markov's inequality therefore yields
\[
 \sup_{T\ge1}\Pi_T^{(x_0,y_0)}({V_0}>R)
 \le \frac{{V_0}(x_0,y_0)+C}{R}
 \rightarrow0\qquad\text{as }R\to\infty.
\]
Each set \(\{V_0\le R\}\) is compact in \(E\). Hence
\(\{\Pi_T^{(x_0,y_0)}:T\ge1\}\) is tight.
The invariance of every weak limit as \(T\to\infty\)
follows from \Cref{lem:time-average-limits}.
\end{proof}

\begin{proof}[Proof of \Cref{thm:imm-existence}]
By \Cref{lem:imm-time-averages}, the time-averaged transition
laws have a weakly convergent subsequence, and every such
limit is an invariant probability.
For any invariant probability \(\pi\), the bounds
in \eqref{eq:imm-moment} follow from
\Cref{lem:common-stationary-moments}, after substituting
\eqref{eq:weighted-generator-expansion} into its first
inequality and cancelling \(C_0\).
\end{proof}

\subsection{Proofs of Corollaries \ref {cor:stationary-balance} and \ref{cor:individual-interaction-balances}}
We use the local martingale \((M_t)_{t \ge 0}\) defined by \eqref{M} 
to justify taking expectations in the stationary equations.

\begin{lemma}
	\label{lem:stationary-noise-martingale}
	Let \(Z\) be a stationary process with marginal distribution \(\pi\).
	Assume that \eqref{eq:imm-moment} and
	\eqref{eq:noise-moment-order} hold.
Then \((M_t)_{t\ge0}\), defined in \eqref{M}, is an
integrable martingale under \(\Pp_\pi\), and
\(\E_\pi M_t=0\) for every \(t\ge0\).
\end{lemma}

\begin{proof}
Since \(u^q\le1+u^r\) for \(u\ge0\) and \(0\le q\le r\), by
stationarity, \eqref{eq:imm-moment} and
\eqref{eq:noise-moment-order}, we have, for \(j=1,2\),
\[
\begin{aligned}
 \E_\pi\int_0^t
 \bigl(b_{1j}X_s^{r_{1j}}+b_{2j}Y_s^{r_{2j}}\bigr)\dd s=t\int_E
 \bigl(b_{1j}x^{r_{1j}}+b_{2j}y^{r_{2j}}\bigr)\pi(\dd x,\dd y)
 <\infty.
\end{aligned}
\]
Split each stable integral at \(z=1\). By It\^o's isometry,
the Brownian and compensated small-jump terms of $M_t$ are square-integrable
martingales. Since \(\int_1^\infty z\,\mu_i(\dd z)<\infty\),
the preceding bound also makes each compensated large-jump term of $M_t$
are integrable martingale. Thus \((M_t)_{t \ge 0}\) is an integrable martingale
with \(\E_\pi M_t=0\).
\end{proof}

\begin{proof}[Proof of \Cref{cor:stationary-balance}]
\textup{(i)} Let \(Z_0\) have law \(\pi\) and $V_0$ define in \eqref{SVW}.
The moment bound \eqref{eq:imm-moment} and \eqref{eq:weighted-generator-expansion} imply \({V_0}\in L^1(\pi)\).
Moreover, \(\iota_1+c_1\iota_2+C_0-\cL {V_0}\) is nonnegative
by \eqref{eq:weighted-dissipation} and integrable by
\Cref{lem:common-stationary-moments}; hence \(\cL {V_0}\in L^1(\pi)\). Taking expectations in \eqref{Vito} under
\(\Pp_\pi\), using stationarity and
\Cref{lem:stationary-noise-martingale}, gives
\[
 0=t\int_E\cL {V_0}(x,y)\pi(\dd x,\dd y),\qquad t>0.
\]
Then \eqref{eq:stationary-balance} follows by \eqref{eq:weighted-generator-expansion}.

\textup{(ii)} Under \eqref{eq:balanced-interactions}, choose
\(c_1=a_1/(-a_2)\) as in the proof of \Cref{lem:parameter-drift}.
The interaction terms cancel, giving part \textup{(ii)}.
\end{proof}

\begin{proof}[Proof of \Cref{cor:individual-interaction-balances}]
As shown in the proof of
\Cref{cor:stationary-balance}, \(\cL V_0\in L^1(\pi)\).
Together with \eqref{eq:weighted-generator-expansion} and
\eqref{eq:imm-moment}, this implies
\[
a_1x^{\theta_1}y^{\kappa_1}
 +c_1a_2x^{\kappa_2}y^{\theta_2}\in L^1(\pi).
\]
Since \(a_1a_2>0\), the two terms have the same sign, so each is integrable. The proof of \Cref{lem:stationary-noise-martingale} shows that
each term of $M_t$ (defined in \eqref{M}) is a martingale.
Taking expectations in the two equations of \eqref{eq:model}
under \(\Pp_\pi\), and using stationarity, gives
\eqref{eq:individual-interaction-balances}.
\end{proof}

\section{Stationary coexistence for square-root diffusions}
\label{sec:sqrt-proofs}

We prove \Cref{cor:sqrt-main}, beginning with the boundary
properties. For the local comparison and boundary results below,
assume \(a_1,a_2\in\R\), \(\iota_i,b_{i0},b_{i1}>0\),
\(r_{i0}>1\), and \(\theta_i,\kappa_i>0\), \(i=1,2\).
Let \((U_t)_{t \ge 0}\) be the  Cox--Ingersoll--Ross (CIR) diffusion satisfying
\[
 U_t=u+\int_0^t(a-cU_s)\dd s
 +\int_0^t\sqrt{2bU_s}\dd B_s,
 \qquad u\ge0,
\]
where \(a,b>0\) and \(c\ge0\), and \((B_t)_{t \ge 0}\) is a Brownian motion.
The CIR--Bessel correspondence and the transition densities in
\cite[Section~1 and Appendix~A.2]{GoingJaeschkeYor2003} give
\begin{equation}
 \Pp_u(U_t=0)=0,\qquad u\ge0,\quad t>0.
 \label{eq:CIR-no-atom}
\end{equation}
Moreover, if \(a\ge b\), the corresponding squared Bessel
process has dimension \(2a/b\ge2\) and is positive at every
positive time, see \cite[Section~1]{GoingJaeschkeYor2003}.
Hence
\begin{equation}
 \Pp_u(U_t>0\text{ for every }t>0)=1,\qquad u\ge0.
 \label{eq:CIR-pathwise-positive}
\end{equation}
We next compare each coordinate with a CIR diffusion up to
\(\tau_N^+\).

\begin{lemma}
\label{lem:localized-cir-comparison}
Let \(Z\) be a nonnegative solution to
\eqref{eq:sqrt-system} with initial state \((x_0,y_0)\in E\).
{Assume that {\(a_1>0>a_2\)},
\eqref{eq:balanced-interactions} holds, and
\(\kappa_1,\kappa_2>0\).}  For \(N>1+({x_0\vee y_0})\), set $\tau_N^+:=\inf\{t\ge0:X_t\vee Y_t\ge N\}$ and $c_{1,N}:=b_{10}N^{r_{10}-1}.$
Then there exists \(c_{2,N}>0\) such that the solutions {\(U^{(i,N)}=(U_t^{(i,N)})_{t\ge0}\), \(i=1,2\)} of
\[
\left\{
\begin{aligned}
 U_t^{(1,N)}
 &=x_0+\int_0^t\bigl(\iota_1-c_{1,N}U_s^{(1,N)}\bigr)\dd s
 +\int_0^t\sqrt{2b_{11}U_s^{(1,N)}}\dd B_1(s),\\
 U_t^{(2,N)}
 &=y_0+\int_0^t\left(\frac{\iota_2}{2}-c_{2,N}U_s^{(2,N)}\right)\dd s
 +\int_0^t\sqrt{2b_{21}U_s^{(2,N)}}\dd B_2(s),
\end{aligned}
\right.
\]
obey
\begin{equation}
 X_t\ge U_t^{(1,N)},
 \qquad Y_t\ge U_t^{(2,N)},
 \qquad 0\le t<\tau_N^+,
 \quad\text{almost surely}.
 \label{eq:cir-localized-comparison}
\end{equation}
If, in addition, \(\kappa_1\ge1\), set
{\(\widetilde c_{2,N}:=(-a_2)N^{\kappa_1+\kappa_2-1}
 +b_{20}N^{r_{20}-1}\)}, and let
\(\widetilde U^{(2,N)}:=(\widetilde U_t^{(2,N)})_{t\ge0}\) solve
\[
 \widetilde U_t^{(2,N)}=y_0
 +\int_0^t\bigl(\iota_2-\widetilde c_{2,N}\widetilde U_s^{(2,N)}\bigr)\dd s
 +\int_0^t\sqrt{2b_{21}\widetilde U_s^{(2,N)}}\dd B_2(s),
\]
then
\begin{equation}
 Y_t\ge\widetilde U_t^{(2,N)},
 \qquad 0\le t<\tau_N^+,
 \quad\text{almost surely}.
 \label{eq:cir-strengthened-comparison}
\end{equation}
\end{lemma}

\begin{proof}
For \(t<\tau_N^+\), we have $X_t^{r_{10}}\le N^{r_{10}-1}X_t,$
and hence the first drift in \eqref{eq:sqrt-system} is bounded below by
\(\iota_1-c_{1,N}X_t\).  For \(0\le v\le N\), set {$h_N(v):=(-a_2)N^{\kappa_2}v^{\kappa_1}+b_{20}v^{r_{20}}.$}
Since \(h_N\) is continuous and \(h_N(0)=0\), choose
\(\delta_N\in(0,N)\) such that $h_N(v)\le\frac{\iota_2}{2}$ for $0\le v\le\delta_N,$
and set $c_{2,N}:=\max_{\delta_N\le v\le N}\frac{h_N(v)}{v}.$
It follows that $h_N(v)\le\frac{\iota_2}{2}+c_{2,N}v$
for $0\le v\le N,$
so the second drift {in \eqref{eq:sqrt-system}} is bounded below by
\(\iota_2/2-c_{2,N}Y_t\) before \(\tau_N^+\). Moreover, \eqref{eq:cir-localized-comparison} follows by the comparison theorem, see, e.g., 
\cite[Theorem~1.1 and Remark~1.1]{IkedaWatanabe1977}.

If \(\kappa_1\ge1\), then, for \(0\le u,v\le N\), we have $v^{\kappa_1}u^{\kappa_2}
 \le N^{\kappa_1+\kappa_2-1}v$ and $v^{r_{20}}\le N^{r_{20}-1}v.$
Thus the second drift in \eqref{eq:sqrt-system} is bounded below by
\(\iota_2-\widetilde c_{2,N}Y_t\) before \(\tau_N^+\).  Then {\eqref{eq:cir-strengthened-comparison}} follows by the same
comparison theorem, see, e.g.,
\cite[Theorem~1.1 and Remark~1.1]{IkedaWatanabe1977}. 
\end{proof}

\begin{lemma}
\label{lem:generalized-cir-comparison}
Let \(Z\) be a nonnegative solution to
\eqref{eq:sqrt-system} with initial state \((x_0,y_0)\in E\).
Assume that \(\theta_i\ge1\), \(i=1,2\), and allow
{\(a_1,a_2\in\R\)}. For \(N>1+(x_0\vee y_0)\), set
\[
 d_{1,N}:=b_{10}N^{r_{10}-1}+(-a_1)^+N^{\theta_1+\kappa_1-1},
 \qquad
 d_{2,N}:=b_{20}N^{r_{20}-1}+(-a_2)^+N^{\theta_2+\kappa_2-1}.
\]
Let \(V^{(1,N)}\) and \(V^{(2,N)}\) solve
\[
\left\{
\begin{aligned}
 V_t^{(1,N)}
 &=x_0+\int_0^t\bigl(\iota_1-d_{1,N}V_s^{(1,N)}\bigr)\dd s
 +\int_0^t\sqrt{2b_{11}V_s^{(1,N)}}\dd B_1(s),\\
 V_t^{(2,N)}
 &=y_0+\int_0^t\bigl(\iota_2-d_{2,N}V_s^{(2,N)}\bigr)\dd s
 +\int_0^t\sqrt{2b_{21}V_s^{(2,N)}}\dd B_2(s).
\end{aligned}
\right.
\]
Then
\begin{equation}
 X_t\ge V_t^{(1,N)},\qquad Y_t\ge V_t^{(2,N)},
 \qquad 0\le t<\tau_N^+,\quad\text{almost surely}.
 \label{eq:cir-generalized-comparison}
\end{equation}
\end{lemma}

\begin{proof}
For \(0\le x,y\le N\),
\[
\begin{aligned}
 \iota_1+a_1x^{\theta_1}y^{\kappa_1}-b_{10}x^{r_{10}}
 &\ge\iota_1-(-a_1)^+N^{\theta_1+\kappa_1-1}x
                 -b_{10}N^{r_{10}-1}x
 =\iota_1-d_{1,N}x,\\
 \iota_2+a_2y^{\theta_2}x^{\kappa_2}-b_{20}y^{r_{20}}
 &\ge\iota_2-(-a_2)^+N^{\theta_2+\kappa_2-1}y
                 -b_{20}N^{r_{20}-1}y
 =\iota_2-d_{2,N}y.
\end{aligned}
\]
Similar to the proof of \Cref{lem:localized-cir-comparison}, the result follows.
\end{proof}

The preceding comparison lemmas and the boundary properties
of CIR diffusions yield positivity at each fixed time \eqref{eq:CIR-no-atom} and,
under \eqref{eq:Feller-thresholds}, simultaneously at all
positive times.

\begin{proposition} 
\label{prop:sqrt-boundary}
{Let \(Z=(X,Y)\) be a nonnegative, nonexplosive solution
of \eqref{eq:sqrt-system} starting from \((x_0,y_0)\in E\).}
\begin{enumerate}[label=\textup{(\roman*)}]
\item {If either {\(a_1>0>a_2\)},
\eqref{eq:balanced-interactions} holds, and
\(\kappa_1,\kappa_2>0\), or \(\theta_i\ge1\) for \(i=1,2\), then}
\begin{equation}
 \Pp(X_t>0,Y_t>0)=1,\qquad t>0.
 \label{eq:fixed-time-interior}
\end{equation}
\item {Assume that either {\(a_1>0>a_2\)},
\eqref{eq:balanced-interactions} holds, and \(\kappa_1\ge1\),
or \(\theta_i\ge1\) for \(i=1,2\).
If \eqref{eq:Feller-thresholds} holds, then}
\begin{equation}
 {\Pp(X_t>0,Y_t>0\text{ for every }t>0)=1.}
 \label{eq:pathwise-interior}
\end{equation}
\end{enumerate}
\end{proposition}

\begin{proof}
\textup{(i)} Fix \(t>0\) and \(N>1+(x_0\vee y_0)\).
In the case \(a_1>0>a_2\), under \eqref{eq:balanced-interactions},
\eqref{eq:cir-localized-comparison} and \eqref{eq:CIR-no-atom} give
\[
 \Pp(X_t=0,\ \tau_N^+>t)
 \le\Pp(U_t^{(1,N)}=0)=0,
\]
and the same argument applies to \(Y_t\).
When \(\theta_i\ge1\), \(i=1,2\), these estimates follow
instead from \eqref{eq:cir-generalized-comparison}, using
\(V^{(1,N)}\) and \(V^{(2,N)}\).
Thus, in either case,
\[
 \Pp(X_t=0\text{ or }Y_t=0)\le\Pp(\tau_N^+\le t).
\]
Since the solution is nonexplosive, \(\lim_{N\to\infty}\tau_N^+=\infty\)
almost surely. Letting \(N\to\infty\) proves
\eqref{eq:fixed-time-interior}.

\textup{(ii)} In the case \(a_1>0>a_2\), use the lower
processes \(U^{(1,N)}\) and \(\widetilde U^{(2,N)}\) from
\eqref{eq:cir-localized-comparison} and
\eqref{eq:cir-strengthened-comparison}.
When \(\theta_i\ge1\), \(i=1,2\), use
\(V^{(1,N)}\) and \(V^{(2,N)}\) from
\eqref{eq:cir-generalized-comparison}.
In both cases, the lower CIR processes have immigration
rates \(\iota_i\) and diffusion coefficients
\(v\mapsto\sqrt{2b_{i1}v}\).
By \eqref{eq:Feller-thresholds} and \eqref{eq:CIR-pathwise-positive},
they remain positive at every positive time almost surely.
The corresponding comparisons therefore give
\begin{equation*}
 \Pp(X_t>0,Y_t>0\text{ for every }0<t<\tau_N^+)=1.
\end{equation*}
Letting \(N\to\infty\) and using nonexplosion proves
\eqref{eq:pathwise-interior}.
\end{proof}

For the remainder of this section, assume
\(\iota_1,\iota_2>0\), \eqref{eq:imm-regularity} and
\Cref{ass:weighted-dissipation},
with \(b_{12}=b_{22}=0\), \(r_{11}=r_{21}=1\), and
\(b_{11},b_{21}>0\). By \Cref{lem:imm-lyapunov}, the solution
of \eqref{eq:sqrt-system} is nonexplosive. For the next two lemmas,
write the drift and diffusion coefficients of
\eqref{eq:sqrt-system} as
\[
 b(x,y):=
 \begin{pmatrix}
 \iota_1+a_1{x^{\theta_1}y^{\kappa_1}}-b_{10}x^{r_{10}}\\
 \iota_2+a_2{y^{\theta_2}x^{\kappa_2}}-b_{20}y^{r_{20}}
 \end{pmatrix},
 \qquad
 \sigma(x,y):=
 \begin{pmatrix}
 \sqrt{2b_{11}x}&0\\
 0&\sqrt{2b_{21}y}
 \end{pmatrix}.
\]
On \(E^\circ\), both coefficients are smooth, and
\[
 \mathsf a(x,y):=\sigma(x,y)\sigma(x,y)^\top
 =\begin{pmatrix}
 2b_{11}x&0\\
 0&2b_{21}y
 \end{pmatrix}
\]
{is uniformly elliptic on every compact
\(K\subset E^\circ\), i.e., there exists \(\eta_K>0\) such that}
\[
 {w^\top\mathsf a(x,y)w\ge\eta_K|w|^2,
 \qquad (x,y)\in K,\quad w\in\R^2.}
\]
{Here \(\top\) denotes transpose.}  For an open set
\(G\subset E^\circ\) and
\((u,v)\in G\), let
\[
\tau_G^{(u,v)}:=\inf\{s\ge0:
 (X_s^{(u,v)},Y_s^{(u,v)})\notin G\}.
\]
Here \((X^{(u,v)},Y^{(u,v)})\) denotes the solution
{of \eqref{eq:sqrt-system}} starting from
\((u,v)\); {under \(\Pp_{u,v}\), we write
\(\tau_G:=\tau_G^{(u,v)}\).}

We next prove the strong Feller property on \(E^\circ\) by
localization to bounded domains.  We call \(E^\circ\) invariant if
\(P_t((x_0,y_0),E^\circ)=1\) for every \((x_0,y_0)\in E^\circ\)
and \(t\ge0\).  On this set, the strong Feller property means that
\(P_tf\in C_b(E^\circ)\) for every \(t>0\) and
\(f\in B_b(E^\circ)\).

\begin{lemma}\label{lem:sqrt-interior-strong-feller}
	 {Assume \eqref{eq:imm-regularity} and \eqref{eq:Feller-thresholds}.}
	Then \(E^\circ\) is invariant, and the restriction of
	\((P_t)_{t\ge0}\) to \(E^\circ\) is strong Feller.
\end{lemma}

\begin{proof}
Since \eqref{eq:imm-regularity} gives \(\theta_i\ge1\),
\(i=1,2\), \Cref{prop:sqrt-boundary}\textup{(ii)} implies that
\(E^\circ\) is invariant.
  Fix
\(t>0\), \(({x_0,y_0})\in E^\circ\), and
\((x_n,y_n)\to({x_0,y_0})\) in \(E^\circ\).
Let \((X^{(n)},Y^{(n)})\) and \(Z\) solve
\eqref{eq:sqrt-system} with initial states \((x_n,y_n)\) and
\((x_0,y_0)\), respectively.
Let \((D_m)_{m\ge1}\) be a sequence of bounded smooth domains such that
\(D_m\Subset D_{m+1}\Subset E^\circ\) and
\(\bigcup_{m\ge1}D_m=E^\circ\).
{By Proposition \ref{prop:sqrt-boundary} and
nonexplosion, the exit times \(\tau_m:=\tau_{D_m}^{(x_0,y_0)}\)
satisfy \(\tau_m\uparrow\infty\) almost surely as $m \rightarrow \infty$. Fix \(\varepsilon>0\). Choose \(m\) large
enough such that \((x_0,y_0)\in D_m\) and}
\begin{equation}
 \Pp_{{x_0,y_0}}(\tau_m\le t)<\varepsilon.
 \label{eq:exit-xy-small}
\end{equation}
Choose a bounded smooth domain \(D'\) with
\(\overline D_m\subset D'\Subset E^\circ\).
Choose a smooth function \(\chi\) with compact support in \(E^\circ\),
\(0\le\chi\le1\), and \(\chi=1\) on a neighborhood of
\(\overline{D'}\).  Set
\[
 \widetilde b:=\chi b,
 \qquad \widetilde\sigma:=\chi\sigma+(1-\chi)I_2
 \quad\text{on }E^\circ,
\]
and extend these coefficients by \(0\) and \(I_2\), respectively,
outside \(E^\circ\), where \(I_2\) is the \(2\times2\) identity 
matrix. Then \(\widetilde b,\widetilde\sigma\in C_b^\infty(\R^2)\),
and they agree with \(b,\sigma\) near \(\overline{D'}\).
  Since \(\sigma\) is
diagonal with positive entries on the compact support of \(\chi\),
there is a constant \(\lambda>0\) such that
\[
 \widetilde\sigma(w)\widetilde\sigma(w)^\top\ge\lambda I_2,
 \qquad w\in\R^2.
\]
For an initial state \((u,v)\in\R^2\), let
\((\widetilde X,\widetilde Y)\) solve
{
\begin{equation}
\begin{aligned}
 \begin{pmatrix}\widetilde X_t\\\widetilde Y_t\end{pmatrix}
 = \begin{pmatrix}u\\v\end{pmatrix}
 +\int_0^t\widetilde b(\widetilde X_s,\widetilde Y_s)\dd s +\int_0^t\widetilde\sigma(\widetilde X_s,\widetilde Y_s)
 \begin{pmatrix}\dd B_1(s)\\\dd B_2(s)\end{pmatrix},
\end{aligned}
 \label{eq:sqrt-modified-system}
\end{equation}
and denote its}
transition semigroup by \((\widetilde P_t)_{t\ge0}\).
  Then by the smoothness, boundedness, and uniform
ellipticity of these coefficients, one sees that
$\widetilde P_t$ maps $B_b(\R^2)$ to $C_b(\R^2)$ for every \(t>0\), see
\cite[Theorem~7.2.4]{StroockVaradhan2006}.

Choose a smooth domain \(D''\) with
\(\overline D_m\subset D''\Subset D'\).
{The coefficients are Lipschitz on \(\overline{D'}\).
By the Burkholder--Davis--Gundy inequality and Gronwall's
inequality, one obtains}
\begin{equation}
 \E\sup_{0\le s\le t\wedge\tau_{D'}^{(x_n,y_n)}
                         \wedge\tau_{D'}^{({x_0,y_0})}}
 \bigl|(X_s^{(n)},Y_s^{(n)})-Z_s\bigr|^2
 \le C_{D',t}\bigl|(x_n,y_n)-({x_0,y_0})\bigr|^2,
 \label{eq:interior-flow-stability}
\end{equation}
{where \(C_{D',t}>0\) is independent of \(n\).}
For all sufficiently large \(n\), \((x_n,y_n)\in D_m\).
On the event
\(\{\tau_{D'}^{(x_n,y_n)}\le t<\tau_{D''}^{({x_0,y_0})}\}\),
set \(s:=\tau_{D'}^{(x_n,y_n)}\).
Then \(s\le t<\tau_{D''}^{({x_0,y_0})}
\le\tau_{D'}^{({x_0,y_0})}\), since \(D''\subset D'\), and hence
\(s=t\wedge\tau_{D'}^{(x_n,y_n)}
\wedge\tau_{D'}^{({x_0,y_0})}\).
By continuity of the sample paths,
\[
 (X_s^{(n)},Y_s^{(n)})\in\partial D',
 \qquad Z_s\in D''.
\]
Therefore, on \(\{\tau_{D'}^{(x_n,y_n)}\le t<
\tau_{D''}^{(x_0,y_0)}\}\),
\beqnn
 \sup_{0\le r\le t\wedge\tau_{D'}^{(x_n,y_n)}
                        \wedge\tau_{D'}^{(x_0,y_0)}}
 \bigl|(X_r^{(n)},Y_r^{(n)})-(X_r,Y_r)\bigr|
 \ge\bigl|(X_s^{(n)},Y_s^{(n)})-Z_s\bigr| \ge d,
\eeqnn
where \(d:=\operatorname{dist}(\overline{D''},\partial D')>0\).
Consequently, by Markov's inequality
 and
\eqref{eq:interior-flow-stability}, for all sufficiently large
\(n\), we have
\begin{equation}
 \Pp_{x_n,y_n}(\tau_{D'}\le t)
 \le\Pp_{x_0,y_0}(\tau_{D''}\le t)
 +\frac{C_{D',t}}{d^2}|(x_n,y_n)-(x_0,y_0)|^2.
 \label{eq:nearby-exit-bound}
\end{equation}
Since \(D_m\subset D''\subset D'\), \eqref{eq:exit-xy-small} yields
\[
 \Pp_{x_0,y_0}(\tau_{D'}\le t)
 \le\Pp_{x_0,y_0}(\tau_{D''}\le t)
 \le\Pp_{x_0,y_0}(\tau_m\le t)<\varepsilon.
\]
Together with \eqref{eq:nearby-exit-bound}, this implies, for all
sufficiently large \(n\),
\begin{equation}
 \Pp_{x_n,y_n}(\tau_{D'}\le t)<2\varepsilon.
 \label{eq:nearby-exit-small}
\end{equation}
Let \(f\in B_b(E^\circ)\) and extend it to a bounded Borel function
\(\widetilde f\) on \(\R^2\) with the same supremum norm.
For a common initial state \((u,v)\in D'\), the solutions of
\eqref{eq:sqrt-system} and \eqref{eq:sqrt-modified-system} satisfy
\[
 Z_s=(\widetilde X_s,\widetilde Y_s),
 \qquad 0\le s\le\tau_{D'},
\]
by pathwise uniqueness and continuity, since
\(b=\widetilde b\) and \(\sigma=\widetilde\sigma\) on \(D'\).
Recall that \(\widetilde f=f\) on \(E^\circ\) and
\(\|\widetilde f\|_\infty=\|f\|_\infty\).
Thus, on \(\{\tau_{D'}\ge t\}\),
\(f(Z_t)=\widetilde f(\widetilde X_t,\widetilde Y_t)\).
Consequently,
\begin{equation}
\begin{aligned}
 |P_tf(u,v)-\widetilde P_t\widetilde f(u,v)|
 &=\left|\E_{u,v}\left[
 \bigl(f(Z_t)-\widetilde f(\widetilde X_t,\widetilde Y_t)\bigr)
 \1_{\{\tau_{D'}<t\}}\right]\right|\\
 &\le2\|f\|_\infty\,\Pp_{u,v}(\tau_{D'}\le t).
\end{aligned}
\label{eq:localized-semigroup-bound}
\end{equation}
Since \(\widetilde P_t\widetilde f\in C_b(\R^2)\),
\begin{align*}
  |P_tf(x_n,y_n)-P_tf(x_0,y_0)|&\le |P_tf(x_n,y_n)-\widetilde P_t\widetilde f(x_n,y_n)|\\
 &\quad+|\widetilde P_t\widetilde f(x_n,y_n)
            -\widetilde P_t\widetilde f(x_0,y_0)| +|\widetilde P_t\widetilde f(x_0,y_0)-P_tf(x_0,y_0)|.
\end{align*}
The middle term tends to zero as \(n\to\infty\). By
\eqref{eq:exit-xy-small}, \eqref{eq:nearby-exit-small}, and
\eqref{eq:localized-semigroup-bound}, the first and last terms
are bounded by \(4\varepsilon\|f\|_\infty\) and
\(2\varepsilon\|f\|_\infty\), respectively, for all sufficiently
large \(n\). Hence
\[
 \limsup_{n\to\infty}|P_tf(x_n,y_n)-P_tf(x_0,y_0)|
 \le6\varepsilon\lVert f\rVert_\infty.
\]
Letting \(\varepsilon\downarrow0\) proves the assertion.
\end{proof}

The next lemma shows that every nonempty open subset of \(E^\circ\) is reached with positive probability at each positive time.

\begin{lemma} 
\label{lem:sqrt-interior-irreducibility}
For every \(t>0\), \((x_0,y_0)\in E^\circ\), and nonempty open
set \(O\subset E^\circ\), we have \(P_t((x_0,y_0),O)>0\).
\end{lemma}

\begin{proof}
Fix \(({x_0,y_0})\in E^\circ\), a nonempty open set
\(O\subset E^\circ\), and \(t>0\).
Choose \(q\in O\) and set
\[
 \phi(s):=\left(1-\frac{s}{t}\right)(x_0,y_0)
             +\frac{s}{t}q,\qquad 0\le s\le t.
\]
Define
\[
b^{\mathrm S}(x,y)
:=b(x,y)-\frac12\begin{pmatrix}b_{11}\\b_{21}\end{pmatrix}.
\]
Since \(b_{11},b_{21}>0\), the matrix \(\sigma(x,y)\)
is invertible on \(E^\circ\). Define
\begin{equation}
 u(s):=\sigma(\phi(s))^{-1}
       \bigl(\dot\phi(s)-b^{\mathrm S}(\phi(s))\bigr) 
 \label{eq:support-control}
\end{equation}
By compactness, there exist \(0<\varepsilon<N<\infty\) such that
\(\phi([0,t])\subset[\varepsilon,N]^2\).
The inverse of \(\sigma\) is continuous and bounded on
\([\varepsilon,N]^2\).
  Moreover, both \(\dot\phi\) and
\(b^{\mathrm S}\circ\phi\) are continuous and bounded on \([0,t]\).
Hence \(u\) is bounded, and therefore
\[
 \int_0^t |u(s)|^2\dd s
 \le t\sup_{0\le s\le t}|u(s)|^2<\infty.
\]
 By \eqref{eq:support-control}, 
\[
 {\dot\phi(s)
 =b^{\mathrm S}(\phi(s))+\sigma(\phi(s))u(s),
 \qquad 0\le s\le t.}
\]
Choose a bounded domain \(D\) with smooth boundary such that
\(\phi([0,t])\subset D\Subset E^\circ\).
As in the proof of \Cref{lem:sqrt-interior-strong-feller},
construct coefficients \(\widetilde b,\widetilde\sigma\in
C_b^\infty(\R^2)\), with uniformly elliptic covariance,
that agree with \(b,\sigma\) near \(\overline D\).
Let \(\widetilde Z := (\widetilde X,\widetilde Y)\) denote the solution to 
\eqref{eq:sqrt-modified-system} with 
initial state \((x_0,y_0)\), and the Brownian motions
\(B_1,B_2\) driving \(Z\).
Choose \(\delta>0\) such that
\[
 \delta<\operatorname{dist}\bigl(\phi([0,t]),\R^2\setminus D\bigr),
 \qquad B(q,\delta)\subset O.
\]
The path \(\phi\) satisfies the controlled equation for the
extended coefficients as well. Thus the support theorem
\cite[Theorem~3.5]{MilletSanzSole1994}, applied with
\(h(s):=\int_0^s u(r)\dd r\) and a deterministic rescaling
of \([0,t]\) to \([0,1]\), gives
\begin{equation}
 \Pp_{{x_0,y_0}}\left(
 \sup_{0\le s\le t}
 |(\widetilde X_s,\widetilde Y_s)-\phi(s)|<\delta
 \right)>0.
 \label{eq:support-event}
\end{equation}
On the event in \eqref{eq:support-event},
\(\widetilde Z_s\in B(\phi(s),\delta)\subset D\)
for every \(0\le s\le t\). Since the coefficients agree on \(D\), local pathwise uniqueness
gives \(Z_s= \widetilde Z_s\) for all
\(s\in[0,t]\) on this event. In particular,
\[
 Z_t= \widetilde Z_t 
 \in B(\phi(t),\delta)=B(q,\delta)\subset O.
\]
The event in \eqref{eq:support-event} has positive probability.
Hence \(P_t((x_0,y_0),O)>0\).
\end{proof}

We now use \Cref{prop:sqrt-boundary},
\Cref{lem:sqrt-interior-strong-feller,lem:sqrt-interior-irreducibility},
and \Cref{thm:imm-existence}
and \Cref{cor:stationary-balance}
to prove \Cref{cor:sqrt-main}.

\begin{proof}[Proof of \Cref{cor:sqrt-main}]
\textup{(i)}
The square-root model \eqref{eq:sqrt-system} is the special case of
\eqref{eq:model} with \(r_{11}=r_{21}=1\) and \(b_{12}=b_{22}=0\).
Existence and the moment bounds follow from
\Cref{thm:imm-existence}. Since
\(r_{11}=r_{21}=1<r_{i0}\), \(i=1,2\),
\eqref{eq:stationary-balance} follows by \Cref{cor:stationary-balance}.
To prove the remaining assertions, let \(\pi\) be
any invariant probability and fix \(t>0\).
By \Cref{prop:sqrt-boundary}\textup{(i)},
\(P_t((x_0,y_0),E^\circ)=1\) for every
\((x_0,y_0)\in E\). Consequently,
\[
 \pi(E^\circ)
 =\int_EP_t((x,y),E^\circ)\pi(\dd x,\dd y)=1.
\]
Let \(G=E\cap O\ne\varnothing\), where \(O\) is open in \(\R^2\).
Since \(E^\circ\) is dense in \(E\), \(G\cap E^\circ\) is a nonempty open
subset of \(E^\circ\). Recall that \(\pi(E^\circ)=1\). By
\Cref{lem:sqrt-interior-irreducibility},
\[
 \pi(G)\ge\pi(G\cap E^\circ)
 =\int_{E^\circ}P_t((u,v),G\cap E^\circ)\pi(\dd u,\dd v)>0
\]
for any \(t>0\). Since \(G\) was arbitrary, we conclude that
\(\supp\pi=E\).

 \textup{(ii)}
We first prove uniqueness. Existence follows from \textup{(i)}.
Let \(\pi_1\) and \(\pi_2\) be invariant probabilities on \(E\).
By \textup{(i)}, both assign mass one to \(E^\circ\).
Since \(E^\circ\) is invariant by
\Cref{lem:sqrt-interior-strong-feller}, both probabilities are
invariant for the restricted semigroup on \(E^\circ\).
The restricted semigroup is strong Feller by
\Cref{lem:sqrt-interior-strong-feller}. It is also irreducible:
\Cref{lem:sqrt-interior-irreducibility} gives
\(P_t(z,O)>0\) for every \(t>0\), \(z\in E^\circ\),
and nonempty open \(O\subset E^\circ\). Thus every
\(P_t(z,\cdot)\) has full support in \(E^\circ\).
To apply Doob's uniqueness theorem, we show that, for each
\(t>0\), the measures \(P_{2t}(z,\cdot)\), \(z\in E^\circ\),
are mutually absolutely continuous. Fix a Borel set
\(A\subset E^\circ\). If
\(P_{2t}((x,y),A)=0\) for some \((x,y)\in E^\circ\), then
\[
 0=\int_{E^\circ}P_t((u,v),A)\,P_t((x,y),\dd u,\dd v).
\]
The integrand is continuous and nonnegative by the strong Feller
property, while \(P_t((x,y),\cdot)\) has full support on
\(E^\circ\). Hence \(P_t((u,v),A)=0\) for every
\((u,v)\in E^\circ\). It follows that all the measures
\(P_{2t}((u,v),\cdot)\), \((u,v)\in E^\circ\), are
equivalent. By \cite[Theorem~1, p.~84]{KulikScheutzow2015},
the chain obtained by observing the process at times
\(0,2t,4t,\ldots\) has at most one invariant probability.
Both \(\pi_1\) and \(\pi_2\) are invariant for this chain,
so \(\pi_1=\pi_2\).

We next prove convergence of the time-averaged transition laws.
Let \(\pi\) be the unique invariant probability.
By \Cref{lem:imm-time-averages}, the family
\(\{\Pi_T^{(x_0,y_0)}:T\ge1\}\) is tight, and every weak
subsequential limit is invariant. Hence every such limit equals
\(\pi\), and \(\Pi_T^{(x_0,y_0)}\Rightarrow\pi\) as
\(T\to\infty\).
\end{proof}

\appendix
\section{Existence, uniqueness, and boundary continuation}
\label{sec:appendix}

\begin{proposition}
\label{prop:axis-realization}
Assume that \(\iota_1=\iota_2=0\), {\(a_1,a_2\in\R\)},
and \eqref{eq:boundary-compatible} holds.
Then, for every initial state \((x_0,y_0)\in E\), the construction
in \Cref{def:coordinatewise} defines the CWP \(Z\) as a strong Markov process with lifetime \(\tau_\infty\).
If \eqref{eq:model} is source-conservative, then
\(\Pp_{x_0,y_0}(\tau_\infty=\infty)=1\) for every
\((x_0,y_0)\in E\).
\end{proposition}

\begin{proof}
	First suppose that \((x_0,y_0)\in E^\circ\). For all sufficiently
	large integers \(n\), let
	\(p_n(u):=(u\vee n^{-1})\wedge n\) and
	\(D_n:=(n^{-1},n)^2\).
	Replace each state variable \(u\) in \eqref{eq:model} by
	\(p_n(u)\). The truncated drift is bounded and Lipschitz. Moreover, for \(u,v\in\R\) and \(i=1,2\), we have
	\(b_{i2}|p_n(u)^{r_{i2}}-p_n(v)^{r_{i2}}|
	\le C_n|u-v|\) and
	\[
	\bigl|\sqrt{2b_{i1}p_n(u)^{r_{i1}}}
	-\sqrt{2b_{i1}p_n(v)^{r_{i1}}}\bigr|^2
	\le C_n|u-v|^2.
	\]
	Since \(\int_0^\infty(z^2\wedge z)\mu_i(\dd z)<\infty\), by \cite[Lemma~A.1]{JiXiongXuYang2026}, there exists a
	pathwise unique \(\R^2\)-valued strong c\`adl\`ag truncated
	solution \(Z^{(n)}=(X^{(n)},Y^{(n)})\).
	Construct these solutions with the same initial state and
	driving noises. For \(m\ge n\), let
	\(\sigma_n^{(m)}:=\inf\{t\ge0:Z_t^{(m)}\notin D_n\}\).
	Then by pathwise uniqueness, we have $
	\sigma_n^{(m)}=\sigma_n^{(n)}=:\sigma_n,$ and $
Z_{t\wedge\sigma_n}^{(m)}=Z_{t\wedge\sigma_n}^{(n)}$ for $t \ge 0.$
	Consequently, $
	\sigma_n\uparrow\tau_0\wedge\tau_\infty$ and $Z_t:=Z_t^{(n)}$ if $t<\sigma_n$, which
	defines the pathwise unique strong c\`adl\`ag  solution.
	
	For \(i=1,2\) and \(u\ge0\), the one-dimensional localization
	in \cite[Theorem~3.1(i)]{LiYangZhou2019}, with absorption at
	zero, gives the axis process \(U^{(i)}\) satisfying
	\[
		U_t^{(i)}= u-b_{i0}\int_0^t(U_s^{(i)})^{r_{i0}}\dd s
		+\int_0^t\sqrt{2b_{i1}(U_s^{(i)})^{r_{i1}}}\dd B_i(s) +\int_0^t\int_0^\infty
		\int_0^{b_{i2}(U_{s-}^{(i)})^{r_{i2}}}
		z\,\widetilde N_i(\dd s,\dd z,\dd v),
	\]
	up to its explosion time. In fact, by
	\cite[Theorem~2.8(i)]{LiYangZhou2019}, one can see that the axis processes are conservative.
On \(\{\tau_0<\tau_\infty\}\), we restart the axis equations
from \(X_{\tau_0}\) and \(Y_{\tau_0}\) by using the shifted
driving noises, and define
\[
Z_{\tau_0+t}:=(U_t^{(1)},U_t^{(2)}),\qquad t\ge0.
\]
Since zero is absorbing for each axis process, this gives
the continuation in \Cref{def:coordinatewise}.
Boundary initial states are treated with \(\tau_0=0\),
and the origin $(0, 0)$ is absorbing. On
\(\{\tau_\infty\le\tau_0\}\), the process is killed at
\(\tau_\infty\).
By \cite[Theorem~4.4.2]{EthierKurtz1986}, by the uniqueness in law
for the interior and absorbed axis processes, together
with localization, there is a Borel strong Markov family
with lifetime \(\tau_\infty\).
If \eqref{eq:model} is source-conservative, the conservativeness
of the axis processes implies that the continuation is
conservative. This completes the proof.
\end{proof}

\begin{proposition}
\label{prop:imm-wellposed}
Let {\(a_1,a_2\in\R\)}. For \(i=1,2\), suppose that
\(\iota_i\ge0\), \(b_{i0}>0\), \(r_{i0}>1\),
\(\theta_i,\kappa_i\ge1\), and
\(r_{ij}\ge1\) whenever \(b_{ij}>0\), \(j=1,2\).
Then, for every \((x_0,y_0)\in E\),
equation \eqref{eq:model} admits a unique
\(E\)-valued strong c\`adl\`ag solution on
\([0,\tau_\infty)\).
\end{proposition}

\begin{proof}
	Fix \((x_0,y_0)\in E\). For \(m\ge1\), let
	\(p_m(u):=u^+\wedge m\) and consider the truncated equation
	\begin{equation}
		\left\{
		\begin{aligned}
			X_t^{(m)}={}&x_0+\int_0^t
			\bigl[\iota_1+a_1p_m(X_s^{(m)})^{\theta_1}
			p_m(Y_s^{(m)})^{\kappa_1}
			-b_{10}p_m(X_s^{(m)})^{r_{10}}\bigr]\dd s\\
			&+\int_0^t
			\sqrt{2b_{11}p_m(X_{s-}^{(m)})^{r_{11}}}\dd B_1(s)
			+\int_0^t\int_0^\infty
			\int_0^{b_{12}p_m(X_{s-}^{(m)})^{r_{12}}}
			z\,\widetilde N_1(\dd s,\dd z,\dd u),\\
			Y_t^{(m)}={}&y_0+\int_0^t
			\bigl[\iota_2+a_2p_m(Y_s^{(m)})^{\theta_2}
			p_m(X_s^{(m)})^{\kappa_2}
			-b_{20}p_m(Y_s^{(m)})^{r_{20}}\bigr]\dd s\\
			&+\int_0^t
			\sqrt{2b_{21}p_m(Y_{s-}^{(m)})^{r_{21}}}\dd B_2(s)
			+\int_0^t\int_0^\infty
			\int_0^{b_{22}p_m(Y_{s-}^{(m)})^{r_{22}}}
			z\,\widetilde N_2(\dd s,\dd z,\dd u).
		\end{aligned}
		\right.
		\label{eq:imm-truncated-model}
	\end{equation}
The truncated drift, covariance matrix and jump rates are
bounded and continuous, and
\(\int_0^\infty(z^2\wedge z)\mu_i(\dd z)<\infty\).
By \cite[Theorem~2.2]{Stroock1975}, the martingale
problem associated with \eqref{eq:imm-truncated-model}
has a solution on \(D([0,\infty),\R^2)\). By
\cite[Theorem~2.3]{Kurtz2011}, on an enlarged
probability space, there exists an adapted c\`adl\`ag weak solution
\(Z^{(m)}=(X^{(m)},Y^{(m)})\) of
\eqref{eq:imm-truncated-model}.
	For each active Brownian coefficient and \(u,v\in[0,R]\),
	\beqnn
		\bigl|\sqrt{2b_{i1}u^{r_{i1}}}
		-\sqrt{2b_{i1}v^{r_{i1}}}\bigr|^2
		\le2b_{i1}|u^{r_{i1}}-v^{r_{i1}}|
		\le C_R|u-v|.
		\eeqnn
	Moreover, there exists a constant $C_m > 0$ such that $
	2b_{i1}p_m(u)^{r_{i1}}\le C_m u^+$ for $u \in \R$.
	Since \(\int_{0+}(C_m u)^{-1}\dd u=\infty\), by the
	zero-local-time criterion, one can check that
	\(L_t^0(X^{(m)})=L_t^0(Y^{(m)})=0\).
	On \(\{X_{s-}^{(m)}\le0\}\), the noise coefficients vanish
	and the drift equals \(\iota_1\).
	All jumps of \(X^{(m)}\) are nonnegative, and no jump occurs
	when \(X_{s-}^{(m)}\le0\).
By Tanaka's
	formula, almost surely for every \(t\ge0\),
	\begin{align*}
		0\le(X_t^{(m)})^-
		&=(x_0)^-
		-\int_0^t\1_{\{X_{s-}^{(m)}\le0\}}\dd X_s^{(m)}
		+\frac12L_t^0(X^{(m)})\\
		&\quad+\sum_{0<s\le t}
		\Bigl[(X_s^{(m)})^--(X_{s-}^{(m)})^-
		+\1_{\{X_{s-}^{(m)}\le0\}}
		(X_s^{(m)}-X_{s-}^{(m)})\Bigr]\\
		&=-\iota_1\int_0^t
		\1_{\{X_{s-}^{(m)}\le0\}}\dd s\le0.
	\end{align*}
	Thus \((X_t^{(m)})^-=0\) for every \(t\ge0\), almost surely.
	The same argument gives \((Y_t^{(m)})^-=0\).	Thus \(Z^{(m)}\) takes values in \(E\).
	
	We next prove pathwise uniqueness. The truncated drift is
	globally Lipschitz, and, for \(u,v\in\R\),
	\[
	b_{i2}|p_m(u)^{r_{i2}}-p_m(v)^{r_{i2}}|
	\le C_m|u-v|.
	\]
	Let \(Z^{(m)}\) and \(\widetilde Z^{(m)}\) be two solutions
	of \eqref{eq:imm-truncated-model} driven by the same noises, and let $\Delta X^{(m)}:=X^{(m)}-\widetilde X^{(m)},$ $
	\Delta Y^{(m)}:=Y^{(m)}-\widetilde Y^{(m)}$ and $D_t^{(m)}:=|\Delta X_t^{(m)}|+|\Delta Y_t^{(m)}|.$
	By the same zero-local-time criterion, one can check that
	\(L_t^0(\Delta X^{(m)})=L_t^0(\Delta Y^{(m)})=0\).
	Applying Tanaka's formula and taking expectations, as in
	\cite[Proposition~3.1]{LiMytnik2011}, we have
	\[
	\E D_t^{(m)}\le D_0^{(m)}
	+C_m\int_0^t\E D_s^{(m)}\dd s.
	\]
It follows that $\E D_t^{(m)}\le\e^{C_mt}D_0^{(m)}$.
	By \cite[Theorems~1 and~2]{BarczyLiPap2015},
	\eqref{eq:imm-truncated-model} therefore has a pathwise unique
	adapted strong c\`adl\`ag solution for every \(m\).
	
	For later use, the same argument applies to two solutions
	\(Z=(X,Y)\) and \(\widetilde Z=(\widetilde X,\widetilde Y)\)
	of \eqref{eq:model}, driven by the same noises, with initial
	states \((x_0,y_0)\) and \((\widetilde x_0,\widetilde y_0)\).
	For
	\(R>x_0\vee y_0\vee\widetilde x_0\vee\widetilde y_0\), let
	\[
	\sigma_R:=\inf\{t\ge0:
	X_t\vee Y_t\vee\widetilde X_t\vee\widetilde Y_t\ge R\},
	\qquad D_t:=|X_t-\widetilde X_t|+|Y_t-\widetilde Y_t|.
	\]
	By Gronwall's inequality, it gives that
	\begin{equation}
		\E D_{t\wedge\sigma_R}\le\e^{C_Rt}D_0,
		\qquad t\ge0.
		\label{eq:stopped-L1-continuity}
	\end{equation}
	In particular, pathwise uniqueness holds up to the explosion
	time. 
As in the proof of \Cref{prop:axis-realization},
we construct all \(Z^{(m)}\) with the same initial state and
driving noises, and $\rho_n^{(m)}:=\inf\{t\ge0:
X_t^{(m)}\vee Y_t^{(m)}\ge n\}$ for $m\ge n>x_0\vee y_0.$
By \eqref{eq:stopped-L1-continuity},
\(\rho_n^{(m)}=\rho_n^{(n)}\) and the stopped solutions
are consistent. The same patching argument shows that $
\tau_\infty:=\lim_{n\to\infty}\rho_n^{(n)}$ and $Z_t:=Z_t^{(n)}$ if $t<\rho_n^{(n)}$
defines an \(E\)-valued strong c\`adl\`ag solution of
\eqref{eq:model} on \([0,\tau_\infty)\). Then
pathwise uniqueness follows from
\eqref{eq:stopped-L1-continuity} by letting \(R\to\infty\).
\end{proof}

\end{document}